\documentclass[reqno]{amsart}

\usepackage[T1]{fontenc}
\usepackage[utf8]{inputenc}
\usepackage{microtype}
\usepackage{amsmath,amsthm,amssymb}
\usepackage[at]{easylist}
\usepackage{constants}
\usepackage{enumerate}
\usepackage{color}
\allowdisplaybreaks

\renewconstantfamily{normal}{symbol=\gamma, format=\arabic}\newconstantfamily{eps}{symbol=C, format=\arabic}
\newconstantfamily{i}{symbol=\mathcal{I}, format=\arabic}
\newconstantfamily{j}{symbol=\mathcal{J}, format=\arabic}

\theoremstyle{plain}
\newtheorem{thm}{Theorem}[section]

\newtheorem{lemma}[thm]{Lemma}
\newtheorem{corollary}[thm]{Corollary}
\newtheorem{proposition}[thm]{Proposition}
\newtheorem{remark}[thm]{Remark}

\let\TeXchi\chi
\newbox\chibox
\setbox0 \hbox{\mathsurround0pt $\TeXchi$}
\setbox\chibox \hbox{\raise\dp0 \box 0 }
\def\chi{\copy\chibox}

\numberwithin{equation}{section}

\newcommand{\field}[1]{\mathbb{#1}}
\newcommand{\R}{\field{R}}

\newcommand{\set}[1]{{\left\{ #1\right\}}}               	
\newcommand{\pa}[1]{{\left(#1\right)}}                  	
\newcommand{\sq}[1]{{\left[#1\right]}}                  	
\newcommand{\Mi}{\mathcal M}
\newcommand{\abs}[1]{\left| #1 \right|}

\newcommand{\norm}[1]{\left\| #1 \right\|}              	

\def\Xint#1{\mathchoice
  {\XXint\displaystyle\textstyle{#1}}%
  {\XXint\textstyle\scriptstyle{#1}}%
  {\XXint\scriptstyle\scriptscriptstyle{#1}}%
  {\XXint\scriptscriptstyle\scriptscriptstyle{#1}}%
  \!\int}
\def\XXint#1#2#3{{\setbox0=\hbox{$#1{#2#3}{\int}$}
  \vcenter{\hbox{$#2#3$}}\kern-.5\wd0}}

\def\YYint#1#2#3{\setbox0=\hbox{$#1{#2#3}{\iint}$}
   \vcenter{\hbox{$#2#3$}}\kern-0.5\wd0}

\newcommand{\st}{\,:\,}                                       	

\newcommand{\eps}{\varepsilon}
\renewcommand{\Lambda}{\varLambda}
\renewcommand{\Delta}{\varDelta}

\newcommand{\loc}{\mathrm{loc}}

\DeclareMathOperator{\diver}{div}
\renewcommand{\div}{\diver}

\DeclareMathOperator*{\esssup}{ess\,sup}

\DeclareMathOperator{\supp}{spt}

\DeclareMathOperator{\dist}{dist}

\newcommand{\vfield}[1]{\mathbf{#1}}
\newcommand{\A}{\vfield{A}}

\newcommand{\spacedot}{\, \cdot \,}

\newcommand{\deq}{\equiv}

\newcommand{\origin}{o}

\newcommand{\Mean}[1]{{(#1)}}

\title[Self-improving property of degenerate equations]{Self-improving property of degenerate parabolic equations of porous medium-type}

\author{Ugo Gianazza}
\address{Dipartimento di Matematica ``F. Casorati''\\
 	Universit\`a di Pavia\\
	via Ferrata 1, 27100 Pavia, Italy}
\email[U.~Gianazza]{gianazza@imati.cnr.it}
\author{Sebastian Schwarzacher}
\address{Katedra matematick\'e anal\'yzy\\ 
Matematicko-fyzik\'aln\'\i fakulta Univerzity Karlovy\\ 
Soko\-lovsk\'a 83\\ 
186 75 Praha 8, Czech Republic}	
\email{schwarz@karlin.mff.cuni.cz}
\usepackage[rose,frak,operator]{paper_diening}

\begin{document}

\begin{abstract}
We show that the gradient of solutions to degenerate parabolic equations of porous medium-type satisfies a reverse H\"older inequality in suitable intrinsic cylinders. {We modify the by-now classical Gehring} lemma by introducing an intrinsic Calder\'on-Zygmund covering argument, and we are able to prove local higher integrability of the gradient {of a proper power of the solution $u$}.
\end{abstract}

\maketitle

\bigskip
\noindent
{\small \emph{Key words and phrases}:
Degenerate parabolic equation, porous medium equation, higher integrability, Gehring's Lemma, Calder\'on-Zygmund covering.
}

\medskip
\noindent
{\small Mathematics Subject Classification (2010):
Primary: 35K65,
Secondary:  35B65.
}

\section{Introduction and main result}

The aim of this paper is to study regularity properties of the gradient of \emph{non-negative} solutions to nonlinear, parabolic, partial differential equations, whose prototype is the \emph{porous medium equation}
\begin{align}\label{PMD-eq: model}\tag{PME}
u_{t} - \Delta u^{m} = u_{t} - \div \pa{mu^{m-1} Du} = 0\qquad m > 0.
\end{align}

When $m = 1$, the nonlinear behavior disappears and \eqref{PMD-eq: model} reduces to the  standard heat equation. When $m \neq 1$, the equation is quasi-linear and its \emph{modulus of ellipticity} is $u^{m-1}$.
When $m > 1$, this quantity  vanishes on the set $[u = 0]$, the time evolution dominates over the diffusion process (\emph{slow diffusion case}) and the equation is said to be \emph{degenerate}. When $0 < m < 1$ the equation becomes \emph{singular}, since the modulus of ellipticity blows up as $u \to 0$; in this case we have the so-called \emph{fast diffusion equation}. 

Equations of this form arise from applications, for example in modelling the flow of an isotropic gas in a porous medium or in studying the heat radiation of plasmas.
{For $m>1$ the \emph{non-linear heat transfer} has a finite speed of propagation, i.e. if the initial datum has a finite support, so does the solution for any positive time. Therefore, it naturally appears in numerous physical models, whenever the assumption of constancy of the thermal conductivity (respectively diffusivity) cannot be sustained. Besides the previously mentioned examples, this is also the case of models of population dynamics (where $u$ describes the concentration of the species) and of the theory of lubrication and boundary layers.}

From the mathematical point of view, the understanding of local behavior of solutions to such equations plays a role in the $C^{1,\alpha}$-regularity theory for systems of $p$-Laplacian type, see for example \cite{DiBenedetto:1993}. 
As can be seen from the so-called Barenbatt fundamental solution (see for example \cite{Bare52}), in general solutions to porous medium equation are considerable less regular than solutions to the parabolic $p$-Laplacian, particularly with respect to gradient estimates, as it shall be discussed with more details below

In this paper we will deal only with the degenerate situation $m >1$, and we will study a general class of equations which have the same structure as \eqref{PMD-eq: model}.

Given a  bounded, open set $E\subset \R^{n}$ with $n\geq2$, and $T >0$, let $E_{T} \deq E \times (0,T)$. 

For $m>1$ and a \emph{positive} right-hand side $f\in L^\frac{m+1}{m-1}_{\loc}(E_T)$, we will consider \emph{non-negative} solutions to
\begin{equation}\label{IPME}
u_{t} -  \div \A(x,t,u,Du) = f \qquad \text{weakly in $E_T$.}
\end{equation} 
The vector field
 \begin{align*}
\A \colon E_T\times \R \times \R^{n} \to \R^{n}
 \end{align*}
 is only assumed to be measurable, and we suppose there exist constants $0 < \nu \leq L < \infty$ such that

\begin{align}\label{PMD-eq: structure}
\begin{aligned}
&\A(x,t,u,Du) \cdot Du \geq \nu\, m{u}^{m-1}\abs{Du}^{2} \\
&\abs{\A(x,t,u,Du)} \leq L\, m{u}^{m-1}\abs{Du}.
\end{aligned} \qquad\qquad\text{ for a.e. }\ (x,t) \in E_T
\end{align}
\vskip.2truecm
As discussed in \cite[Chapter~3, \S~5]{DiBGiaVes11}, the structure conditions \eqref{PMD-eq: structure} are not 
sufficient to characterize parabolic partial 
differential equations. 
The partial differential equation in \eqref{IPME} 
is  \emph{parabolic} if it satisfies \eqref{PMD-eq: structure} and in addition, for every 
weak, local sub(super)-solution $u$ (see the precise definitions below), for all $k\in\setR$, the truncations 
$(u-k)_+$ and  $-(u-k)_-$ are weak, local 
sub(super)-solutions to \eqref{IPME}, 
in the sense of \eqref{PMD-eq: reg weak sol}--\eqref{PMD-eq: weak sol} below, 
with $\mathbf{A}(x,t,u,Du)$ replaced by 
\begin{equation*}
\A(x,t,k\pm(u-k)_\pm,\pm D(u-k)_\pm).
\end{equation*}
In \cite[Chapter~3, \S~5]{DiBGiaVes11} the following result is stated and proved.
\begin{lemma}
Assume that for all 
$(x,t,u)\in E_T\times\setR$
\begin{equation}\label{Eq:1:5:8}
\A(x,t,u,\eta)\cdot\eta\ge0
\quad\text{ for all }\> \eta\in\setR^n.
\end{equation}
Then \eqref{IPME}--\eqref{PMD-eq: structure} is parabolic.
\end{lemma}
Henceforth we will assume that the principal part 
$\A(x,t,u,Du)$  satisfies \eqref{Eq:1:5:8} too,
so that \eqref{IPME}--\eqref{PMD-eq: structure} is parabolic.
\vskip.2truecm
As we have already mentioned, the model problem \eqref{PMD-eq: model} corresponds to the case $\nu = L = 1$ and $f\equiv 0$.
\subsection{Weak solutions and sub(super)-solutions}
A function 
\begin{align}\label{PMD-eq: reg weak sol}
u \in C^0_{\loc} \pa{0,T; L^{2}_{\loc} (E)}\quad \text{with} \quad u^{\frac{m+1}{2}} \in L^{2}_{\loc} \pa{0,T ; W^{1,2}_{\loc}(E)}
\end{align}
is a local, weak sub(super)-solution to to \eqref{IPME}-\eqref{PMD-eq: structure} 
if satisfies the integral identity
\begin{align}\label{PMD-eq: weak super-sol}
\iint_{E_{T}} - u \phi_{t} + \A(x,t,u,Du) \cdot D\phi \,dxdt \le(\ge) \iint_{E_{T}} f\phi\,dxdt
\end{align}
for all possible choices of non-negative test functions $\phi \in C^{\infty}_o(E_{T})$.
This guarantees that all the integrals in \eqref{PMD-eq: weak super-sol} 
are convergent.

A {local, weak solution} to \eqref{IPME}-\eqref{PMD-eq: structure} is both a sub- and a super-solution, i.e., it  satisfies the integral identity
\begin{align}\label{PMD-eq: weak sol}
\iint_{E_{T}} - u \phi_{t} + \A(x,t,u,Du) \cdot D\phi \,dxdt = \iint_{E_{T}} f\phi\,dxdt
\end{align}
for all possible choices of test functions $\phi \in C^{\infty}_o(E_{T})$. 
Moreover, we talk of \emph{homogeneous} equations, whenever $f\equiv0$.

By a standard \emph{mollification} argument, it is possible to use the solution $u$ as test function. Let $\zeta \colon \R \to \R$,
 \begin{align*}
\zeta(s) \deq
\left\{
\begin{aligned}
&\ C \exp \pa{\frac{1}{\abs{s}^{2} - 1}} && \abs{s} < 1\\
&\ 0 && \abs{s} \geq 1
\end{aligned}
\right.
 \end{align*}
be the standard mollifier ($C$ is chosen in order to have $\norm{\zeta}_{L^{1}(\R)} = 1$) and define the \emph{family}
 \begin{align*}
\zeta^{\eps}(s) = \frac{1}{\eps} \zeta \pa{\frac{s}{\eps}}, \qquad \eps > 0.
 \end{align*}
Since we need a \emph{time regularization}, given  $\phi \in C^{\infty}_o( E _{T})$, we consider the family of mollifiers $\set{\zeta^{\eps}}$, with
 \begin{align*}
\eps < \dist \pa{\supp \phi,  E_T},
 \end{align*}
and we set
 \begin{align*}
\phi_{\eps}(x,t) = \pa{\varphi \star \zeta^{\eps}}(x,t) =\int_{\R} \phi(x, t-s) \zeta^{\eps}(s)\,ds.
 \end{align*} 

We insert $\phi_{\eps}$ as test function in \eqref{PMD-eq: weak sol}, change variables and apply Fubini's theorem to obtain
\begin{align}\label{PMD-eq: weak sol mollified}
\iint_{ E _{T}} - u_{\eps} \phi_{t} + \A_{\eps}(x,t,u,Du) \cdot D\phi\,dxdt = \iint_{ E _{T}} f_\eps\phi\,dxdt,
\end{align}
where the subscript in $u_{\eps}$, $f_\eps$, and $\A_{\eps}$ denotes the mollification with respect to time.

We conclude this introductory section with our main result. 
\begin{thm}[Local higher integrability]\label{thm:main}
Let $u \geq 0$ be a local, weak solution to \eqref{IPME}-\eqref{PMD-eq: structure} in $E_T$ for $m>1$.
Then, there exists $\eps_{\origin} > 0$, depending only on $n$, $m$, $\nu$, and $L$ of 
\eqref{PMD-eq: structure}, such that
 \begin{align*}
 u^{\frac{m+1}{2}} \in L^{2 + \eps}_{\loc} \pa{0,T ; W^{1,2 + \eps}_{\loc}(E)} \qquad \forall \eps \in (0, \eps_{\origin}].
 \end{align*}
\end{thm}
Theorem~\ref{thm:main} is a straightforward {consequence of local quantitative estimates. We provide two different versions, a first one for standard parabolic cylinders $B_r(x_o)\times(t_o-r^2,t_o]$ (see Theorem~\ref{thm-para}), and a second version on the so-called \emph{intrinsic} cylinders (see below), which inherit the natural scaling properties of the solution (see Theorem~\ref{thm-intr}).} 
\subsection{Novelty and Significance}
As apparent from the statement of Theorem~\ref{thm:main}, we are interested in the order of \emph{integrability} of $|Du^{\frac{m+1}{2}}|$. 
For \emph{elliptic} equations and systems, Meyers \& Elcrat \cite{Meyers:1975} showed that the gradients of solutions locally belong to a slightly higher Sobolev space than expected a priori. The main tools are a reverse H\"older inequality for $\abs{Du}$ and an application of \emph{Gehring's lemma} (see the original paper \cite{Gehring:1973} and also \cite{Giaquinta:1979,Stredulinsky:1980}).
The method works for equations with $p$-growth, hence degenerate and singular elliptic equations of $p$-Laplacian type are allowed.

Giaquinta \& Struwe \cite{Giaquinta:1982} extended the elliptic, local, higher integrability result to parabolic equations. However, in their work, in order to derive the reverse H\"older inequality, the diffusion term $\A\approx Du$, i.e. it is forced to have {a linear} growth with respect to $\abs{Du}$, so that degenerate and singular equations are ruled out.

The main obstruction to the extension to the degenerate/singular setting is given by the  \emph{lack of homogeneity} in the energy estimates.  This problem can be overcome by using the so-called {\em intrinsic parabolic geometry}, that is a scaling, which depends on the solution itself. Under a more physical point of view, the diffusion process evolves at a time scale which depends instant by instant on $u$ itself; the homogeneity is recovered, once the time variable is rescaled by a factor that depends on the solution in a suitable way.  
This approach was first developed by DiBe\-ne\-detto \& Friedman~\cite{DiBFri85,DiBenedetto:1993} in the context of the parabolic $p$-Laplace equation
\begin{equation}\label{plap}
\begin{aligned}
&u \in C^{0}_{\loc}\pa{0,T; L^{2}_{\loc}(E)} \cap L^{p}_{\loc}\pa{0,T; W^{1,p}_{\loc}(E)}\\
&u_{t} - \div\pa{\abs{Du}^{p-2} Du} = f\qquad \text{weakly in $E_T$,}
\end{aligned}
\qquad p > 1.
\end{equation}

Localisation with respect to intrinsically scaled cylinders, was a key tool to prove the H\"older continuity of the gradients, for smooth right-hand sides $f$.  

Later on, by rephrasing these ideas in the context of intrinsic Calderon-Zygmund coverings, Kinnunen \& Lewis \cite{Kinnunen:2000} showed that gradients {of solutions to equations with the same structure as \eqref{plap}} enjoy a higher integrability property, namely
 \begin{align*}
D u \in L^{p+ \eps}_{\loc}(E_T), \qquad \text{for some }\ \eps > 0.
 \end{align*}
This result holds under very general structural assumptions on the operator, and minimal conditions on the right-hand side. The values of $p$ cover the full degenerate range $p > 2$, but are restricted to the \emph{super-critical} singular range $\frac{2n}{n+2} < p < 2$. {This restriction on $p$ in the singular range is a recurrent feature, as discussed at length, for example in \cite[Appendix~B]{DiBGiaVes11}, or in \cite{Acerbi:07}.} {It is noteworthy that, based on the local higher integrability result of the parabolic $p$-Laplacian, many applications follow. These include (without any ambition of completion) the full $L^q$ theory and beyond~\cite{Acerbi:07,Sch13}, as well as partial regularity results~\cite{BoDuMi:2013}, or pointwise estimates via potential theory for equations~\cite{KuuMing14}. Summarizing, many different ways to show regularity for the gradient of solutions to the $p$-Laplacian are available which, among other benefits, has a natural impact on the regularity of the time derivative, as it was recently shown in~\cite{FreSch15}.}

{Taking into account the large amount of results, it might not seem too surprising that the }
adaption of the non-linear methods developed for the $p$-Laplace to the porous medium equation turns out to be more {delicate} than expected; indeed, the higher integrability result for the porous medium equations has been an open problem for some years. 

{To our knowledge, so far, the only existing, related contribution is due to B{\'e}nilan \cite{Benilan}; he established an abstract result whose application to nonnegative solutions of the porous medium equation yields $\displaystyle\frac{\partial^2 u^m}{\partial x_i\partial x_j}\in L^p_{\loc}(E_T)$ for $1< p<1+1/m$. Since solutions are bounded, by the Sobolev embedding theorem one obtains that 
\[
Du^m\in L^{\bar p}_{\loc}(E_T),\qquad \text{ for }\ \bar p=\frac{np}{n-p}. 
\]
Benilan's solutions satisfy $Du^m(\cdot,t)\in L^2(E)$ for almost every $t\in(0,T)$, and therefore, the previous result amounts to a higher integrability estimate for 
\[
n<\frac{2m+2}{m-1}. 
\]
The method makes no use of reverse 
H\"older inequalities or of Gehring's Lemma, and does not apply to equations with the same generality as considered here: indeed, Theorem~\ref{thm:main} covers a wide class of equations relying only on \eqref{PMD-eq: structure}, and gives a quantitative knowledge (at least theoretically) of the higher integrability of the gradient of proper powers of local solutions.} {Moreover, the estimates developed here are of quantitative and local nature.}

What are the main difficulties in the proof of the higher integrability for solutions to \eqref{IPME}-\eqref{PMD-eq: structure}? How come, it turns out that it is far from being a straightforward extension of techniques already used when dealing with the analogous result for the $p$-Laplace equation? We conclude this section, by shortly discussing these difficulties  and the related technical novelties, which allow to overcome them.

The first step to higher integrability is a {\em reverse H\"older estimate}, which has always been the backbone to gain Gehring type higher integrability properties. The reverse H\"older estimate, which might be of independent interest, is stated in Proposition~\ref{rh:1} for general parabolic cylinders.

However, this estimate alone does not allow to conclude the desired higher integrability, in contrast to 
the $p$-Laplacian situation. Indeed, if $c$ is a positive constant, and $u$ is a local solution to the $p$-Laplace equation~\eqref{plap}, then $u-c$ is a solution as well. Clearly, this is not the case for the porous medium equation and this simple difference makes it impossible to apply to the porous medium equation the approach known for the $p$-Laplacian. Indeed, the latter is based on (scaling invariant) Sobolev-\Poincare inequalities of Gagliardo-Nirenberg type in space time, which are invariant by a constant (i.e. the mean value). 

We overcome this difficulty by splitting the problem into two cases: degenerate and non-degenerate regimes. This is a very common approach, going back to DeGiorgi. It was used to get estimates for solutions of PDEs in many different ways. For the $p$-Laplacian, besides the references already given and without pretending to list all the relevant contributions, we refer for example to~\cite{BoDuMi:2013,BCDKS15,DieLenStoVer11,KuuMin14b,Sch13}. For solutions to~\eqref{PMD-eq: model} it is a key tool to derive Harnack inequalities, as well as to prove H\"older continuity: just as an example, see \cite{DiBGiaVes11}. 
However, as far as we can say, its use in the context of gradient estimates for the porous medium equation is a novelty. 
 
For each case, namely the {\em degenerate} and the {\em non-degenerate} regime, we prove estimates on {\em intrinsic cylinders}, as stated respectively in Proposition~\ref{deg} and~\ref{non-deg}. These are {\em invariant under subtraction of constants}, and therefore suitably tailored to our purposes. The proofs employ tools, which are different in the two cases.

In the degenerate regime, roughly speaking, we use the fact that we have a control on the amount of oscillation which solutions can have.
Conversely, in the non-degenerate case, we rely on the expansion of positivity for super-solutions, as stated in Theorem~\ref{thm:pos}. Note that the use of Theorem~\ref{thm:pos} is the only (but crucial) place, where we need the assumption of parabolicity, as well as the positivity of the right-hand side.

In Section~\ref{sec:int} we prove the higher integrability via an argument of Gehring type. Note that it is not possible to use any standard reference in a straightforward manner. As a matter of fact, we need to establish a \Calderon-Zygmund covering, using cylinders which are intrinsically scaled with respect to what seems to be the natural quantity here, namely $u^{m-1}$, where $u$ is the solution. On the other hand, since we want to estimate proper powers of $\abs{D u^\frac{m+1}{2}}$, we have to develop something like an intrinsic metric with respect to $u$, which can then be used for the \Calderon-Zygmund analysis, {\em independently of the PDE}. In turn, this heavily relies on the possibility of adapting the construction of~\cite{Sch13} to the porous medium equation.

We think that this tool might be of independent interest in the analysis of degenerate partial differential equations, as it translates the intrinsic scaling into a sort of \emph{intrinsic distance}. 
\vskip.1truecm
\noindent{\it Acknowledgements.}
S.~Schwarzacher thanks program PRVOUK P47, financed by the Charles University in Prague. Both authors acknowledge the warm hospitality of the Institut Mittag-Leffler,
where this research project started, during the program ``Evolutionary problems'' in the Fall
2013.
\section{Preliminaries}

\subsection{Notation}
Consider a point $z_{\origin} = (x_{\origin},t_{\origin}) \in \R^{n+1}$
and two parameters $\rho, \tau > 0$. The open ball with radius $\rho$ and center $x_{\origin}$ will be denoted by
 \begin{align*}
B_{\rho}(x_{\origin}) \deq \set{x \in E \st \abs{x-x_{\origin}} < \rho}.
 \end{align*}
We define the time-space cylinder by
 \begin{align*}
Q_{\tau,\rho}(z_{\origin})  \deq \pa{t_{\origin} - \tau, t_{\origin} + \tau}\times B_{\rho}(x_{\origin}) .
 \end{align*}
As we prove local estimates, the reference point is never of importance, and we often omit it by writing $B_{\rho}$ and $Q_{\tau,\rho}$.
%

The symbol $\abs{\spacedot}$ stands for the Lebesgue measure, either in $\R^{n}$ or $\R^{n+1}$, and the dimension will be clear from the context.

\subsection{Constants and data}
As usual, the letter $c$ is reserved to positive constants, whose value may change from line to line, or even in the same formula.
We say that a generic constant $c$ \emph{depends on the data}, if $c = c(n, m, \nu , L)$,
where $\nu$ and $L$ are the quantities introduced in \eqref{PMD-eq: structure}.

Let $\eta\in L^\infty(E)$: we denote by 
 \begin{align*}
(g)_E^\eta=\frac{1}{\norm{\eta}_{L^1(E)}}\int_E g\,\eta\, dx.
 \end{align*}
In the special case of {$\eta\equiv1$}, we write
 \begin{align*}
(g)_E^1=:(g)_E=: \dashint_Eg\, dx.
 \end{align*}
We will frequently use what we will refer to in the following as \emph{the best constant property}. For positive $\eta$ we have, for any $c\in\setR$ and $q\in [1,\infty)$
\begin{align}
\label{bcp}
\bigg(\frac{1}{\norm{\eta}_{L^1(E)}}\int_E\abs{g-(g)_E^\eta}^q\eta\,dx\bigg)^\frac1q\leq 2\bigg(\frac{1}{\norm{\eta}_{L^1(E)}}\int_E\abs{g-c}^q\eta\,dx\bigg)^\frac1q.
\end{align}
Moreover, if $0\leq \eta\leq 1$, by \eqref{bcp} one obviously obtains that
\begin{align}\label{meanchange}
\bigg(\frac{1}{\norm{\eta}_{L^1(E)}}\int_E\abs{g-(g)_E^\eta}^q\eta\,dx\bigg)^\frac1q\leq 2\bigg(\frac{1}{\norm{\eta}_{L^1(E)}}\int_E\abs{g-(g)_E}^q dx\bigg)^\frac1q,
\end{align}
and as a consequence, that
\begin{equation}\label{meanchange2}
\begin{aligned}
\abs{(g)_E^\eta-(g)_E}&\leq 
\bigg(\frac{1}{\norm{\eta}_{L^1(E)}}\int_E\abs{g-(g)_E^\eta}^q\eta\,dx\bigg)^\frac1q\\
&\leq 2\bigg(\frac{1}{\norm{\eta}_{L^1(E)}}\int_E\abs{g-(g)_E}^q dx\bigg)^\frac1q. 
\end{aligned}
\end{equation}
We will also use the following estimate that was first proved in \cite[Lemma 6.2]{DieKapSch11}: for $q\geq2$ we have 
\begin{align}
\label{trick}
\dashint_E\abs{g-\Mean{g}_E}^q dx\leq c_o\dashint_E\abs{g^\frac{q}{2}-(g)_E^\frac{q}{2}}^2 dx\leq c_1\dashint_E\abs{g^\frac{q}{2}-(g^\frac{q}{2})_E}^2 dx,
\end{align}
{where $c_o$ and $c_1$ are constants that depend only on the data.}

\section{Reverse H\"older Inequalities in General Cylinders}
{The main result of this section is Proposition~\ref{rh:1}, an estimate of reverse 
H\"older type on general cylinders: the second and third term on its right-hand side, characterized by the multiplying factor $\tilde\delta$, can be seen as error terms. In \S~\ref{S:intrinsic}, relying on this first result, and working in suitably scaled intrinsic cylinder (see \eqref{def:intr} for the definition of such an object), we will  prove proper reverse 
H\"older estimates without error terms.}

Proposition~\ref{rh:1} holds in general cylinders of the form $Q_{\theta\rho^2,\rho}$, which we assume to be ``centered'' at the origin for the sake of simplicity, i.e. $Q_{\theta\rho^2,\rho}=(-\theta\rho^2,0)\times B_\rho$. It results from the combination of the energy estimate of Lemma~\ref{lem:osc-energy}, and the purely analytic estimate of Gagliardo-Nirenberg type of Lemma~\ref{lem:sobpoinc}.

\begin{lemma}\label{lem:osc-energy}
Let $u \geq 0$ be a  local, weak solution to \eqref{IPME}-\eqref{PMD-eq: structure} with $m>1$. For $\rho,\theta > 0$, suppose $
Q_{2\theta\rho^2,2\rho}\subset E_{T}$,
Then, there exists a constant $c > 1$ depending only on the data, such that
\begin{align*}
&\esssup_{t\in(-\theta\rho^2,0]}\dashint_{\set{t}\times B_{\rho}}\frac{\abs{u-(u(t))_{B_{\rho}}}^2}{\theta\rho^2}\,dx + \dashint_{-\theta\rho^2}^{0}\dashint_{B_{\rho}}  u^{m-1}\abs{Du}^{2}\,dxdt \\
&\quad\leq c\dashint_{-2\theta\rho^2}^{0}\dashint_{B_{2\rho}}\left[\Big(u^{m-1}+\frac1{\theta}\Big)
\frac{\abs{u-(u(t))^{{\eta^2}}_{B_{2\rho}}}^{2}}{\rho^2}+\frac{\abs{u-(u(t))^{{\eta^2}}_{B_{2\rho}}}^{m+1}}{\rho^2}+\rho^\frac{2}{m} f^\frac{m+1}{m}\right]\,dxdt.
\end{align*}
Moreover, for $1\leq a<b\leq 2$, we have 
\begin{equation}\label{eq:en-est-2}
\begin{aligned}
&\esssup_{t\in(-a\theta\rho^2,0]}\dashint_{\set{t}\times B_{a\rho}}\frac{\abs{u-(u(t))_{B_{a\rho}}}^2}{\theta\rho^2}\, dx + \dashint_{-a\theta\rho^2}^0\dashint_{B_{a\rho}}  u^{m-1}\abs{Du}^{2}\,dxdt \\
&\quad\leq \frac{c}{(b-a)^2}\dashint_{-b\theta\rho^2}^0\dashint_{B_{b\rho}}\Big[\Big(u^{m-1}+\frac1{\theta}\Big)
\frac{\abs{u-(u(t))^{{\eta^2}}_{B_{{b}\rho}}}^{2}}{\rho^2}+
\frac{\abs{u-(u(t))^{{\eta^2}}_{B_{{b}\rho}}}^{m+1}}{\rho^2}\\
&+\rho^\frac{2}{m} f^\frac{m+1}{m}\Big]dxdt.
\end{aligned}
\end{equation}
\end{lemma}
\begin{proof}
The first estimate is a consequence of the second one. Therefore, let $1\leq a<b\leq 2$.
 Fix $s\in(-a\frac\theta2\rho^2,0]$ and
 $\eta\in C^{0,1}_0(B_{b\rho})$, with 
 \[
 \abs{D\eta}\leq \frac{c}{(b-a)\rho},\ \ \chi_{B_{a\rho}}\leq\eta\leq\chi_{B_{b\rho}}.
 \]
 Then, take the test function $\phi(x,t)=(u(x,t)-(u(t))_{B_{b\rho}}^{\eta^2})[\eta(x)]^2\left(\frac{t+b\theta\rho^2}{s+b\theta\rho^2}\right)$, where
\begin{align*}
(u(t))_{B_{b\rho}}^{\eta^2}:=\frac{1}{\int_{B_{b\rho}}[\eta(x)]^2\,dx}\int_{\set{t}\times B_{b\rho}}u(x,t)\,[\eta(x)]^2\,dx.
 \end{align*}
We insert the test function in the weak formulation of \eqref{IPME} and find
\begin{align}
\label{eq:energy1}
\int_{-{b}\theta \rho^2}^s\dashint_{B_{b\rho}}[\partial_t u\,\phi+\mathbf A(x,t,u,D u)\cdot D\phi-f\phi]\,dxdt=0.
\end{align}
We estimate each integrand separately. We start with the first one.
Notice that for any measurable function $g:(-2\theta\rho^2,0]\to\setR$, by a formal computation which can be made rigorous by a standard Steklov average, we find that
 \begin{align*}
 \int_{\set{t}\times B_{{b}\rho}}\partial_tu(u-(u(t))_{B_{{b}\rho}}^{\eta^2})\,\eta^2 dx=\int_{\set{t}\times B_{{b}\rho}}(\partial_tu-g(t))(u-(u(t))_{B_{{b}\rho}}^{\eta^2})\,\eta^2 dx,
 \end{align*}
which implies by the right Steklov approximation, that
\begin{align*}
\int_{\set{t}\times B_{b\rho}}\partial_tu(u-(u(t))_{B_{b\rho}}^{\eta^2})\,\eta^2 dx=\int_{\set{t}\times B_{b\rho}}\partial_t\frac{\abs{u-(u(t))_{B_{{b}\rho}}^{\eta^2}}^2}{2}\eta^2 dx.
\end{align*}
Therefore, the time-term can be transformed and estimated using the best constant property in the following way:
\begin{align*}
&\int_{-{b}\theta \rho^2}^s\dashint_{B_{b\rho}} \partial_t u\,\phi\,dxdt\\
&\quad=\dashint_{\set{s}\times B_{b\rho}}\frac{\abs{u-(u(t))_{B_{b\rho}}^{\eta^2}}^2}{2}\eta^2\, dx
-\int_{-b\theta\rho^2}^s\dashint_{\set{s}\times B_{b\rho}}\frac{\abs{u-(u(t))_{B_{{b}\rho}}^{\eta^2}}^2}{2}\eta^2\frac1{s+b\theta\rho^2}\,dxdt\\
&\quad \geq\dashint_{\set{s}\times B_{{a}\rho}}\frac{\abs{u-(u(t))_{B_{b\rho}}^{\eta^2}}^2}{2}\, dx
-2\int_{-b\theta\rho^2}^s\dashint_{\set{s}\times B_{b\rho}}\frac{\abs{u-(u(t))^{{\eta^2}}_{B_{b\rho}}}^2}{(b-a)\theta\rho^2}\eta^2\,dxdt\\
&\quad\geq \frac14\dashint_{\set{s}\times B_{{a}\rho}}\abs{u-(u(t))_{B_{{a}\rho}}}^2\, dx
-2\int_{-b\theta\rho^2}^s\dashint_{\set{s}\times B_{b\rho}}\frac{\abs{u-(u(t))^{{\eta^2}}_{B_{{b}\rho}}}^2}{(b-a)\theta\rho^2}\,dxdt.
\end{align*}
To analyze the second integrand, we find by \eqref{PMD-eq: structure} and Young's inequality for every $t$, that
\begin{align*}
\dashint_{B_{b\rho}}\mathbf A(x,t,u,D u)\cdot D\phi\,dx \geq &\,\nu \dashint_{B_{b\rho}}\abs{D u^\frac{m+1}{2}}^2\eta^2 dx\\
&-{(m-1)L}\dashint_{B_{b\rho}}\abs{D u^\frac{m+1}{2}}\eta u^\frac{m-1}2\abs{u-{(u(t))^{\eta^2}_{B_{b\rho}}}}\abs{D \eta} dx\\
\geq &\,\frac{\nu}{2} \dashint_{B_{b\rho}}\abs{D u^\frac{m+1}{2}}^2\eta^2\,dx-c(L,\nu,m)\dashint_{B_{b\rho}}u^{m-1}\frac{\abs{u-{(u(t))^{\eta^2}_{B_{b\rho}}}}^2}{(b-a)^2\rho^2} dx.
\end{align*}
Finally, the last term on the right-hand side is estimated by Young's inequality for every $t$, as
\begin{align*}
\abs{\dashint_{B_{b\rho}}f\phi\,dx}\leq c\dashint_{B_{b\rho}}\rho^{{\frac{2}{m}}} f^\frac{m+1}{m} dx+ c\dashint_{B_{b\rho}}\frac{\abs{u-{(u(t))^{\eta^2}_{B_{b\rho}}}}^{m+1}}{\rho^2} dx.
\end{align*}
Inserting all these estimates in \eqref{eq:energy1}, yields \eqref{eq:en-est-2}. 
\end{proof}
By a completely analogous argument we have
\begin{lemma}\label{lem:energy}
Let $u \geq 0$ be a  local, weak solution to \eqref{IPME}-\eqref{PMD-eq: structure} with $m>1$. For $\rho,\theta > 0$, suppose $
Q_{2\theta\rho^2,2\rho}\subset E_{T}$,
Then, there exists a constant $c > 1$ depending only on the data, such that
\begin{align*}
&\esssup_{t\in(-\theta\rho^2,0]}\dashint_{\set{t}\times B_{\rho}}\frac{u^2}{\theta\rho^2}\,dx + \dashint_{-\theta\rho^2}^{0}\dashint_{B_{\rho}}  u^{m-1}\abs{D u}^{2}\,dxdt \\
&\quad\leq c\dashint_{-2\theta\rho^2}^{t_o}\dashint_{B_{2\rho}} \left[\Big(u^{m-1}+\frac1{\theta}\Big)
\frac{u^{2}}{\rho^2}+\rho^{{\frac{2}{m}}}f^\frac{m+1}{m}\right]dxdt.
\end{align*}
\end{lemma}
\begin{proof}
It is exactly the same as before; the only difference is in the test function, where the mean values are now discarded.
\end{proof}
The second ingredient is the next lemma, which is purely analytical, independent of any partial differential equation. 

For the sake of simplicity, from here on we let $\Mean{u}_{Q_{\theta\rho^2,4\rho}}=a$ and $\Mean{u}_{\set{t}\times B_{4\rho}}=a(t)$.
\begin{lemma}
\label{lem:sobpoinc}
For any function $u\in L^\infty(-\theta\rho^2,0;L^2(B_{\rho}))$, such that $D u^\frac{m+1}{2}\in L^{2}(Q_{\theta\rho^2,\rho})$, there exist $\gamma\in (0,1)$ and $q_o\in(0,\infty)$, such that for every $\delta\in (0,1)$ there holds
\begin{align*}
 &\dashint_{Q_{\theta\rho^2,\rho}}\frac{\abs{u-(u(t))_{B_{\rho}}}^{m+1}}{\rho^2}\,dxdt
 \\
&\leq \delta \sup_{t\in(-\theta\rho^2,0]}\dashint_{B_{\rho}}\frac{\abs{u-(u(t))_{B_{\rho}}}^2}{\theta\rho^2}dx+c_\delta\big(\theta (u)_{Q_{\theta\rho^2,\rho}}^{m-1}\big)^{q_o}\bigg(\dashint_{Q_{\theta\rho^2,\rho}}\abs{D u^\frac{m+1}{2}}^{2\gamma}\, dxdt\Bigg)^\frac{1}{\gamma},
\end{align*}
where $c_\delta$ depends on $\delta$ and the data only. 
\end{lemma}

\begin{proof}
The proof is done by interpolation. For $b>d>1$ arbitrary and $\beta=\frac{2}{d(m+1)}$ we can find  
$\sigma\in(0,1)$, such that 
$\frac{\sigma}{\beta}+\frac{(1-\sigma)d}{b}=1$. Therefore, 
\begin{align*}
\bigg(\dashint_{B_{\rho}} \abs{u-a(t)}^{(m+1)d}dx\bigg)^\frac{1}{d}
&\leq\bigg(\dashint_{B_{\rho}}{\abs{u-a(t)}^2}dx\bigg)^\frac{\sigma (m+1)}{2}\bigg(\dashint_{B_{\rho}}\abs{u-a(t)}^{(m+1)b}dx\bigg)^\frac{(1-\sigma)}{b}.
\end{align*}
We choose $\alpha\in (0,1)$ such that $\alpha (m+1)+\frac{1-\alpha}{d}=1$, and find
\begin{align*}
 &\dashint_{Q_{\theta\rho^2,\rho}}\abs{u-a(t)}^{m+1}dxdt\leq \bigg(\dashint_{Q_{\theta\rho^2,\rho}}\abs{u-a(t)}dxdt\bigg)^{\alpha(m+1)}\\
 &\times\bigg(\dashint_{Q_{\theta\rho^2,\rho}}\abs{u-a(t)}^{(m+1)d}dxdt\bigg)^\frac{(1-\alpha)}{d}\\
&\leq c\, a^{\alpha(m+1)}\sup_{t\in(-\theta\rho^2,0]}\bigg(\dashint_{B_{\rho}}\abs{u-a(t)}^2dx\Bigg)^\frac{\sigma(m+1)(1-\alpha)}{2}\\
&\times\rho^{2(1-\sigma)(1-\alpha)}\bigg(\dashint_{-\theta\rho^2}^0\bigg(\dashint_{B_{\rho}}\frac{\abs{u-a(t)}^{(m+1)b}}{\rho^{2b}}\, dx\bigg)^\frac{(1-\sigma)d}{b}\, dt\bigg)^\frac{1-\alpha}{d}.
\end{align*}
By \eqref{trick} and \Poincare's inequality we find that
\begin{align*}
 &\dashint_{Q_{\theta\rho^2,\rho}}\abs{u-a(t)}^{m+1}dxdt\leq c a^{\alpha(m+1)}\rho^{2(1-\sigma)(1-\alpha)}\sup_{t\in(-\theta\rho^2,0]}\bigg(\dashint_{B_{\rho}}\abs{u-a(t)}^2dx\bigg)^\frac{\sigma(m+1)(1-\alpha)}{2}
 \\
 &\quad\times
 \bigg(\dashint_{Q_{\theta\rho^2,\rho}}\abs{D u^\frac{m+1}{2}}^{2\gamma}\, dxdt\Bigg)^\frac{(1-\sigma)(1-\alpha)}{\gamma},
\end{align*}
provided that we choose $\gamma\in(0,1)$ such that $\frac{(1-\sigma)d}{\gamma}\leq 1$.
H\"older's inequality for $(1-\sigma)(1-\alpha)+{\sigma(1-\alpha)+\alpha}=1$ gives
\begin{align*}
 &\dashint_{Q_{\theta\rho^2,\rho}}\frac{\abs{u-a(t)}^{m+1}}{\rho^2}dxdt
 \leq c \rho^{-2(\sigma(1-\alpha)+\alpha)}\theta^{-\alpha\frac{m+1}{m-1}} \sup_{t\in(-\theta\rho^2,0]}\bigg(\dashint_{B_{\rho}}\abs{u-a(t)}^2dx\bigg)^\frac{\sigma(m+1)(1-\alpha)}{2} 
 \\
 &\quad\times
 \big(\theta a^{m-1}\big)^{\alpha\frac{m+1}{m-1}}
 \bigg(\dashint_{Q_{\theta\rho^2,\rho}}\abs{D u^\frac{m+1}{2}}^{2\gamma}\, dxdt\Bigg)^\frac{(1-\sigma)(1-\alpha)}{\gamma},
 \\
& \leq \delta \rho^{-2} \theta^{-\big(\frac{m+1}{m-1}\big)\big(\frac{\alpha}{\sigma(1-\alpha)+\alpha}\big)}
\sup_{t\in(-\theta\rho^2,0]}\bigg(\dashint_{B_{\rho}}\abs{u-a(t)}^2dx\bigg)^{\big(\frac{m+1}{2}\big)\big(\frac{\sigma(1-\alpha)}{\sigma(1-\alpha)+\alpha}\big)}\\
 &\quad+c_\delta\big(\theta a^{m-1}\big)^{\frac{\alpha}{(1-\sigma)(1-\alpha)}\frac{m+1}{m-1}}\bigg(\dashint_{Q_{\theta\rho^2,\rho}}\abs{D u^\frac{m+1}{2}}^{2\gamma}\, dxdt\Bigg)^\frac{1}{\gamma}
\end{align*}
Therefore, we need that
 $\frac{\sigma(1-\alpha)}{\sigma(1-\alpha)+\alpha}=\frac{2}{m+1}$ and $\frac{\alpha}{\sigma(1-\alpha)+\alpha}=\frac{m-1}{m+1}$, which can be realized, provided $b$ and $d$ are properly chosen. Finally
 $\displaystyle q_o={\frac{\alpha}{(1-\sigma)(1-\alpha)}\frac{m+1}{m-1}}$.
\end{proof}

\noindent Combining the last two lemmas yields the first estimate of reverse H\"older type.
\begin{proposition}
\label{rh:1}
Let $u \geq 0$ be a  local, weak solution to \eqref{IPME}-\eqref{PMD-eq: structure} with $m>1$. For $\rho,\theta > 0$, suppose $
Q_{2\theta\rho^2,2\rho}\subset E_{T}$.
Then, there exist $\gamma\in (0,1)$ and $q_o>1$, such that for any $\tilde\delta$ there is a constant $c > 1$ depending only on $\tilde\delta$, $\gamma$ and the data, such that
\begin{align*}
&\esssup_{t\in(-\theta\rho^2,0]}\dashint_{\set{t}\times B_{\rho}}\frac{\abs{u-(u(t))_{B_{\rho}}}^2}{\theta\rho^2}\,dx +  \dashint_{Q_{\theta\rho^2,\rho}} \abs{Du^\frac{m+1}{2}}^{2}\,dxdt \\
&\quad\leq c\big(\theta (u)_{Q_{2\theta\rho^2,2\rho}}^{m-1}\big)^{q_o}\bigg(\dashint_{Q_{2\theta\rho^2,2\rho}} \abs{Du^\frac{m+1}{2}}^{2\gamma}dxdt\bigg)\frac1\gamma+\tilde\delta\dashint_{Q_{2\theta\rho^2,2\rho}} \frac{u^{m+1}}{\rho^2}\,dxdt +\frac{\tilde\delta}{\rho^2\theta^\frac{m+1}{m-1}}\\
&+c\dashint_{Q_{2\theta\rho^2,2\rho}}\rho^{{\frac{2}{m}}}f^\frac{m+1}{m}dxdt.
\end{align*}
\end{proposition}
\begin{proof}
Take $1\leq a<b\leq 2$ and estimate the right-hand side of \eqref{eq:en-est-2} by Young's inequality and \eqref{meanchange2}. We have
\begin{align*}
\esssup_{t\in(-a\theta\rho^2,0]}&\dashint_{\set{t}\times B_{a\rho}}\frac{\abs{u-(u(t))_{B_{a\rho}}}^2}{\theta\rho^2}\, dx + \dashint_{Q_{a\theta\rho^2,a\rho}}  u^{m-1}\abs{Du}^{2}\,dxdt \\
\le&\frac{c}{(b-a)^2}\dashint_{Q_{b\theta\rho^2,b\rho}} \Big(u^{m-1}+\frac1{\theta}\Big)
\frac{\abs{u-(u(t))^{\eta^2}_{B_{2\rho}}}^{2}}{\rho^2} dxdt\\
&+c\dashint_{Q_{b\theta\rho^2,b\rho}}\left[ \frac{\abs{u-(u(t))^{\eta^2}_{B_{2\rho}}}^{m+1}}{\rho^2}+\rho^{{\frac{2}{m}}}f^\frac{m+1}{m}\right]dxdt\\
\le&\frac{\tilde\delta}{\rho^2(b-a)^{2\frac{m+1}{m-1}}}\bigg(\frac1{\theta^\frac{m+1}{m-1}}+\dashint_{Q_{b\theta\rho^2,b\rho}}{u^{m+1}}dxdt\bigg)\\
&+c\dashint_{Q_{b\theta\rho^2,b\rho}} \frac{\abs{u-(u(t))_{B_{2\rho}}}^{m+1}}{\rho^2}dxdt
+c\dashint_{Q_{b\theta\rho^2,b\rho}}\rho^{{\frac{2}{m}}}f^\frac{m+1}{m}dxdt.
\end{align*}
We can now apply Lemma~\ref{lem:sobpoinc} to find
\begin{align}\label{eq:rh1}
&\esssup_{t\in(-a\theta\rho^2,0]}\dashint_{\set{t}\times B_{a\rho}}\frac{\abs{u-(u(t))_{B_{a\rho}}}^2}{\theta\rho^2}\, dx + \dashint_{Q_{a\theta\rho^2,a\rho}}  u^{m-1}\abs{Du}^{2}\,dxdt\nonumber \\
\leq&\ \frac{\tilde\delta}{\rho^2(b-a)^{2\frac{m+1}{m-1}}}\bigg(\frac1{\theta^\frac{m+1}{m-1}}+\dashint_{Q_{2\theta\rho^2,2\rho}}{u^{m+1}}dxdt\bigg)
+c\dashint_{Q_{b\theta\rho^2,b\rho}}\rho^{{\frac{2}{m}}}f^\frac{m+1}{m} dxdt\nonumber\\
&\ +
\delta \sup_{t\in (-b\theta\rho^2,0]}\dashint_{B_{b\rho}}\frac{\abs{u-(u(t))_{B_{b\rho}}}^2}{\theta\rho^2}dx\\
&\ +c(\delta,\tilde\delta)\big(\theta (u)_{Q_{2\theta\rho^2,2\rho}}^{m-1}\big)^{q_o}\bigg(\dashint_{Q_{2\theta\rho^2,2\rho}}\abs{D u^\frac{m+1}{2}}^{2\gamma}dxdt\Bigg)^\frac{1}{\gamma}\nonumber
\end{align}
Now the interpolation Lemma~6.1 of \cite{Giu03} yields the result.
\end{proof}
\section{Intrinsic Reverse H\"older Inequalities}\label{S:intrinsic}
We will call $Q_{\theta\rho^2,\rho}$ a \emph{$K$-intrinsic cylinder}, if 
\begin{align}
\label{def:intr}
\frac{1}{K\theta}\leq \Mean{u^{m+1}}_{Q_{\theta\rho^2,\rho}}^\frac{m-1}{m+1}\leq \frac{K}{\theta}.
\end{align}
We call $Q_{\theta\rho^2,\rho}$ a \emph{$K$-sub-intrinsic cylinder}, if only the estimate from above holds. 
In the following, we will avoid any reference to the constant $K$, meaning that the cylinders will be either intrinsic or sub-intrinsic for some {$K>1$}. In the next two subsections we will prove reverse 
H\"older inequalities in intrinsic cylinders. We will have to distinguish two different conditions, the so-called \emph{Degenerate} and \emph{Non-Degenerate} Regimes.
%
\subsection{The Degenerate Regime}
Here and in the following we will consider an intrinsic cylinder $Q_{2\theta\rho^2,2\rho}\subset E_T$, where we assume that
\begin{align}
\label{eq:deg1} \bigg(\dashint_{Q_{2\theta\rho^2,2\rho}}\abs{u-\Mean{u}_{Q_{2\theta\rho^2,2\rho}}}^{m+1}dxdt\bigg)^\frac1{m+1}\geq \epsilon \bigg(\dashint_{Q_{2\theta\rho^2,2\rho}}u^{m+1}dxdt\bigg)^\frac1{m+1}
\end{align}  
for some $\epsilon>0$. We denote this condition as the \emph{Degenerate Regime}. We need the following lemma, {which is proved} using the weak time-derivative of the solution. 
\begin{lemma}\label{lem:averages}
Let $u \geq 0$ be a local, weak solution to {\eqref{IPME}-\eqref{PMD-eq: structure}} for $m > 1$. If $Q_{ \tau,2 \rho}(t_o,x_o) \subset E_{T}$, then there exists $c = c(data)$ such that, for almost all $t_{1}, t_{2} \in (t_o-\tau,t_o)$, we have
 \begin{align*}
\abs{(u(t_{1}))_{B_{\rho}}^\eta - (u(t_{2}))_{B_{\rho}}^\eta} \leq c \, \frac{\tau}{\rho}\dashint_{Q_{\tau,\rho}(t_o,x_o)} \abs{Du^{m}}\,dxdt+c\tau\dashint_{Q_{\tau,\rho}(t_o,x_o)}f\,dxdt,
 \end{align*}
for all  $\eta\in C^{0,1}_o(B_{2\rho})$, with $\abs{D\eta}\leq \frac{c}{\rho}$ and $\chi_{B_\rho}\leq\eta\leq\chi_{B_{2\rho}}$.
Moreover, by H\"older's inequality we find
\begin{align*}
&\abs{(u(t_{1}))_{B_{\rho}}^\eta - (u(t_{2}))_{B_{\rho}}^\eta} 
 \leq c\tau\dashint_{Q_{\tau,\rho}(t_o,x_o)}f\,dxdt\\
&\quad +c\,\frac{\tau}{\rho} \bigg(\dashint_{Q_{\tau,\rho}(t_o,x_o)} \abs{Du^\frac{m+1}{2}}^\frac{2(m+1)}{m+3}\,dxdt\bigg)^\frac{m+3}{2(m+1)}\bigg(\dashint_{Q_{\tau,\rho}(t_o,x_o)}u^{m+1}dxdt\bigg)^\frac{m-1}{2(m+1)}.
\end{align*}
\end{lemma}
\begin{proof}
{We estimate formally (but things can be made rigorous by standard arguments)}. 
\begin{align*}
\abs{(u(t_{1}))_{B_{\rho}}^\eta - (u(t_{2}))_{B_{\rho}}^\eta}&=\Bigabs{\frac{1}{\norm{\eta}_{L^1(B_\rho)}}\int_{t_1}^{t_2}\frac{d}{dt}\int_{B_\rho}u\eta\,dxdt}\\
&=\Bigabs{\frac{1}{\norm{\eta}_{L^1(B_\rho)}}\int_{t_1}^{t_2}\skp{\partial_t u}{\eta}\,dt}\\\
&=\Bigabs{\frac{1}{\norm{\eta}_{L^1(B_\rho)}}\int_{t_1}^{t_2}\skp{\mathbf A(x,t,u,D u)}{D\eta}-\skp{f}{\eta}\,dt}\\
&\leq \frac{c\tau}{\rho}\dashint_{Q_{\tau,\rho}(t_o,x_o)} \abs{Du^{m}}\,dxdt+ c\tau\dashint_{Q_{\tau,\rho}(t_o,x_o)}f\,dxdt.
\end{align*}
\end{proof}
\noindent We also need the following analytical remark.
\begin{lemma}
Let $Q_{s,\rho}\subset \setR^{n+1}$.
If for a positive function $g\in L^q(Q_{2s,2\rho})$ {we have} 
\[
\dashint_{Q_{s,\rho}} g^q\,dxdt\leq K\dashint_{Q_{s,\rho}}\abs{g-\Mean{g}_{Q_{s,\rho}}}^q\,dxdt,
\]
then $\forall a\in(1,2]$
\[
\dashint_{Q_{as,a\rho}} g^q\,dxdt\leq c(q)[a^{n+1}(2K+1)+1]\dashint_{Q_{as,a\rho}}\abs{g-\Mean{g}_{Q_{as,a\rho}}}^q\,dxdt.
\]
\end{lemma}
\begin{proof}
By the best constant property we find
\begin{align*}
\dashint_{Q_{as,a\rho}} g^q\,dxdt=&\dashint_{Q_{as,a\rho}}|g-(g)_{Q_{as,a\rho}}+(g)_{Q_{as,a\rho}}-(g)_{Q_{s,\rho}}+(g)_{Q_{s,\rho}}|^q\,dxdt\\
\leq&c(q)\left[\dashint_{Q_{as,a\rho}}|g-(g)_{Q_{as,a\rho}}|^q\,dxdt+|(g)_{Q_{as,a\rho}}-(g)_{Q_{s,\rho}}|^q+\left(\dashint_{Q_{s,\rho}} g\,dxdt\right)^q\right]\\
\leq&c(q)\left[\dashint_{Q_{as,a\rho}} |g-(g)_{Q_{as,a\rho}}|^q\,dxdt+\dashint_{Q_{s,\rho}}|g-(g)_{Q_{as,a\rho}}|^q\,dxdt\right.\\
&\left.+K\dashint_{Q_{s,\rho}}\abs{g-\Mean{g}_{Q_{s,\rho}}}^q\,dxdt\right]\\
\leq&c(q)\left[(2K+1)\dashint_{Q_{s,\rho}}\abs{g-\Mean{g}_{Q_{as,a\rho}}}^q\,dxdt+\dashint_{Q_{as,a\rho}} |g-(g)_{Q_{as,a\rho}}|^q\,dxdt\right]\\
\leq&c(q)[a^{n+1}(2K+1)+1]\dashint_{Q_{as,a\rho}}\abs{g-\Mean{g}_{Q_{as,a\rho}}}^q\,dxdt.
\end{align*}
\end{proof}
Therefore, by making $\epsilon$ a bit smaller, {from \eqref{eq:deg1}, we deduce that} for all $a\in [2,4]$
\begin{align}
\label{eq:deg} \bigg(\dashint_{Q_{a\theta\rho^2,a\rho}}\abs{u-\Mean{u}_{Q_{a\theta\rho^2,a\rho}}}^{m+1}dxdt\bigg)^\frac1{m+1}\geq \epsilon \bigg(\dashint_{Q_{a\theta\rho^2,a\rho}}u^{m+1}\,dxdt\bigg)^\frac1{m+1}.
\end{align}
\begin{proposition}
\label{deg}
Let $u \geq 0$ be a local, weak solution to {\eqref{IPME}-\eqref{PMD-eq: structure}} for $m > 1$, and let $u$ satisfy the degenerate alternative in the intrinsic cylinder $Q_{2\theta\rho^2,2\rho}$. Then there exist $q\in (0,1)$ and a constant $c>0$ depending on $\epsilon$, and the data, such that, if 
\[
\theta(u^{m})_{Q_{4\theta\rho^2,4\rho}}^\frac{m-1}{m}\leq cK,
\]
then
\begin{equation}\label{rh-deg}
\begin{aligned}
&\sup_{t\in (-2\theta\rho^2,0]}\dashint_{B_{2\rho}}\frac{\abs{u-\Mean{u(t)}_{B_{2\rho}}}^2}{\theta\rho^2}\,dx+\dashint_{Q_{2{\theta}\rho^2,2\rho}} \abs{Du^\frac{m+1}{2}}^{2}\,dxdt+\frac{1}{\rho^2}\dashint_{Q_{4\theta\rho^2,4\rho}} u^{m+1}\,dxdt\\
&\quad\leq c(\epsilon, K) \bigg(\dashint_{Q_{4\theta\rho^2, 4\rho}} \abs{Du^\frac{m+1}{2}}^{2q}\,dxdt\bigg)^\frac{1}{q}+\dashint_{Q_{4\theta\rho^2, 4\rho}} \rho^\frac2m f^\frac{m+1}{m}\,dxdt.
\end{aligned}
\end{equation}
\end{proposition}
\begin{proof}
We apply \eqref{eq:rh1} for $2\leq a<b\leq 4$ with $\tilde\delta=1$ to obtain
\begin{align*}
&\esssup_{t\in(-a\theta\rho^2,0]}\dashint_{\set{t}\times B_{a\rho}}\frac{\abs{u-(u(t))_{B_{a\rho}}}^2}{\theta\rho^2}\, dx + \dashint_{Q_{a\theta\rho^2,a\rho}}  u^{m-1}\abs{Du}^{2}\,dxdt \\
\leq&\ \frac{1}{\rho^2(b-a)^{2\frac{m+1}{m-1}}}\bigg(\frac1{\theta^\frac{m+1}{m-1}}+\dashint_{Q_{b\theta\rho^2,b\rho}}{u^{m+1}}dxdt\bigg)
+c\dashint_{Q_{b\theta\rho^2,b\rho}}\rho^{{\frac{2}m}} f^\frac{m+1}{m} dxdt\\
&\ +
\delta \sup_{t\in (-b\theta\rho^2,0]}\dashint_{B_{b\rho}}\frac{\abs{u-(u(t))_{B_{b\rho}}}^2}{\theta\rho^2}dx\\
&\ +c(\delta)\big(\theta (u)_{Q_{4\theta\rho^2,4\rho}}^{m-1}\big)^{q_o}\bigg(\dashint_{Q_{4\theta\rho^2,4\rho}}\abs{D u^\frac{m+1}{2}}^{2\gamma}dxdt\Bigg)^\frac{1}{\gamma}.
\end{align*}
By the intrinsic {nature of the cylinder}, and enlarging the ball,
\begin{align*}
&\esssup_{t\in(-a\theta\rho^2,0]}\dashint_{\set{t}\times B_{a\rho}}\frac{\abs{u-(u(t))_{B_{a\rho}}}^2}{\theta\rho^2}\, dx + \dashint_{Q_{a\theta\rho^2,a\rho}}  u^{m-1}\abs{Du}^{2}\,dxdt \\
&\ \leq \frac{1+cK^{{\frac{m+1}{m-1}}}}{\rho^2(b-a)^{2\frac{m+1}{m-1}}}\dashint_{Q_{b\theta\rho^2,b\rho}}{u^{m+1}}\,dxdt
+c\dashint_{Q_{2\theta\rho^2,2\rho}}\rho^{{\frac{2}m}}f^\frac{m+1}{m}\,dxdt\\
&\ +
\delta \sup_{t\in (-b\theta\rho^2,0]}\dashint_{B_{b\rho}}\frac{\abs{u-(u(t))_{B_{b\rho}}}^2}{\theta\rho^2}dx+c(\delta)\big(\theta (u)_{Q_{4\theta\rho^2,4\rho}}^{m-1}\big)^{q_o}\bigg(\dashint_{Q_{4\theta\rho^2,4\rho}}\abs{D u^\frac{m+1}{2}}^{2\gamma}\, dxdt\Bigg)^\frac{1}{\gamma}.
\end{align*} 
We are left with the estimate of $\displaystyle\frac1{\rho^2}\dashint_{Q_{b\theta\rho^2,b\rho}}{u^{m+1}}\,dxdt$, which also justifies the extra term in \eqref{rh-deg}, as we make no use of the smallness of $\tilde\delta$ in this proof. Once more, we define $(u(t))_{B_{b\rho}}=a(t)$ and $a=(u)_{Q_{b\theta\rho^2,b\rho}}$. 
Using the degenerate condition~\eqref{eq:deg}, and the best constant property~\eqref{bcp}, we find that
\begin{align*}
&\frac{1}{\rho^2}\dashint_{Q_{b\theta\rho^2,b\rho}} u^{m+1}\,dxdt\leq \frac{1}{\epsilon\rho^2} \dashint_{Q_{b\theta\rho^2,b\rho}}\abs{u-a}^{m+1}\,dxdt \\
&\leq \frac{c}{\epsilon\rho^2}\dashint_{Q_{b\theta\rho^2,b\rho}}\abs{u-(u)_{Q_{b\theta\rho^2,b\rho}}^\eta}^{m+1}\,dxdt\\
&\leq 
\frac{c}{\epsilon\rho^2}\dashint_{Q_{b\theta\rho^2,b\rho}}\abs{u-(u(t))_{B_{b\theta\rho^2,b\rho}}^\eta}^{m+1}\,dxdt+
\frac{c}{\epsilon\rho^2}\dashint_{-b\theta\rho^2}^0\abs{(u(t))_{B_{b\theta\rho^2,b\rho}}^\eta-(u)_{Q_{b\theta\rho^2,b\rho}}^\eta}^{m+1}\,dt,
\end{align*}
where $\eta\in C^\infty_o(B_{2\rho})$ is chosen as in Lemma~\ref{lem:averages}; notice that $(u)_{Q_{b\theta\rho^2,b\rho}}^\eta=(u)_{Q_{2\theta\rho^2,2\rho}}^\eta$. By Jensen's inequality and Lemma~\ref{lem:averages}, 
\begin{align*}
 &\dashint_{-b\theta\rho^2}^0\abs{(u(t))_{B_{b\theta\rho^2,b\rho}}^\eta-(u)_{Q_{b\theta\rho^2,b\rho}}^\eta}^{m+1}dt\\
 &\quad\leq c(u^{m})_{b\theta\rho^2,b\rho}\sup_{t\in (-b\theta\rho^2,0]}\abs{(u(t))_{B_{b\theta\rho^2,b\rho}}^\eta-(u)_{Q_{b\theta\rho^2,b\rho}}^\eta}\\
&\quad \leq c(u^{m})_{Q_{b\theta\rho^2,b\rho}}\theta\rho\bigg(\dashint_{Q_{ b\theta\rho^2,b\rho}} \abs{Du^\frac{m+1}{2}}^\frac{2(m+1)}{m+3}\,dxdt\bigg)^\frac{m+3}{2(m+1)}\bigg(\dashint_{Q_{b\theta\rho^2,b\rho}}u^{m+1}dxdt\bigg)^\frac{m-1}{2(m+1)}\\
&\qquad
+c(u^{m})_{Q_{b\theta\rho^2,b\rho}}\theta\rho^2\dashint_{Q_{b\theta\rho^2,b\rho}}f\,dxdt
\\
&\quad \leq cK(u^{m})_{Q_{b\theta\rho^2,b\rho}}^\frac{1}{m}\rho\bigg(\dashint_{Q_{b\theta\rho^2,b\rho}} \abs{Du^\frac{m+1}{2}}^\frac{2(m+1)}{m+3}\,dxdt\bigg)^\frac{m+3}{2(m+1)}\bigg(\dashint_{Q_{b\theta\rho^2,b\rho}}u^{m+1}dxdt\bigg)^\frac{m-1}{2(m+1)}\\
&\qquad
+cK(u^{m})_{Q_{b\theta\rho^2,b\rho}}^\frac1m\rho^2\dashint_{Q_{b\theta\rho^2,b\rho}}f\,dxdt
\\
&\quad \leq cK\rho\bigg(\dashint_{Q_{b\theta\rho^2,b\rho}} \abs{Du^\frac{m+1}{2}}^\frac{2(m+1)}{m+3}\,dxdt\bigg)^\frac{m+3}{2(m+1)}\bigg(\dashint_{Q_{b\theta\rho^2,b\rho}}u^{m+1}dxdt\bigg)^\frac{1}{2}\\
&\qquad+cK(u^{m})_{Q_{b\theta\rho^2,b\rho}}^\frac1m\rho^2\dashint_{Q_{b\theta\rho^2,b\rho}}f\,dxdt
\\
&\quad \leq c(\delta^*,K)\rho^2\dashint_{Q_{b\theta\rho^2,b\rho}}\bigg(\abs{Du^\frac{m+1}{2}}^{2q}\,dxdt
\bigg)^\frac1q
+\delta^* \dashint_{Q_{b\theta\rho^2,b\rho}}u^{m+1}dxdt\\
&\qquad+c(\delta^*,K)\rho^\frac{2(m+1)}{m}\dashint_{Q_{b\theta\rho^2,b\rho}}f^\frac{m+1}{m}\,dxdt,
\end{align*}
where we can pick $q\in[\frac{m+1}{m+3},1)$. Choose $\delta^*=\frac12$: together with \eqref{meanchange}, this yields
 \begin{align*}
& \frac{1}{\rho^2}\dashint_{Q_{b\theta\rho^2,b\rho}} u^{m+1}\,dxdt\leq \frac{2}{\epsilon\rho^2} \dashint_{Q_{b\theta\rho^2,b\rho}}\abs{u-a(t)}^{m+1}\,dxdt\\
& +\frac{c_\delta((u^{m})_{Q_{b\theta\rho^2,b\rho}}^\frac{m-1}{m}\theta)^2}{\epsilon} \bigg(\dashint_{Q_{b\theta\rho^2,b\rho}} \abs{Du^\frac{m+1}{2}}^{2q}\,dxdt\bigg)^\frac{1}{q}  
+ {\frac{c(\delta,K)}{\epsilon}}\rho^\frac{2}{m}\dashint_{Q_{b\theta\rho^2,b\rho}}f^\frac{m+1}{m}\,dxdt.
\end{align*}
We apply Lemma~\ref{lem:sobpoinc}, where we choose $\delta=\frac{\epsilon(b-a)^2}{4}$ and obtain
\begin{align*}
 \frac{1}{\rho^2}\dashint_{Q_{b\theta\rho^2,b\rho}} u^{m+1}\,dxdt&\leq
 \frac{(b-a)^2}{2}\sup_{t\in (-b\rho^2,0]}\dashint_{B_{b\rho}}\frac{\abs{u-(u(t))_{B_{b\rho}}}^2}{\theta\rho^2}\,dx\\
 &+ \frac{c_{\delta,\epsilon,K}}{(b-a)^{p_o}}\bigg(\bigg(\dashint_{Q_{b\theta\rho^2,b\rho}} \abs{Du^\frac{m+1}{2}}^{2q}\,dxdt\bigg)^\frac{1}{q}+ \rho^\frac{2}{m}\dashint_{Q_{b\theta\rho^2,b\rho}}f^\frac{m+1}{m}\,dxdt\bigg),
 \end{align*}
where the exponent $p_o$ comes from Lemma~\ref{lem:sobpoinc}.
{Collecting all the terms, for $p_1=\max\set{2,p_o}$ we find} 
\begin{align*}
&\sup_{t\in (-\theta a\rho^2,0)}\dashint_{B_{a\rho}}\frac{\abs{u-(u(t))_{B_{a\rho}}}^2}{\theta\rho^2}\,dx+
\dashint_{Q_{2\theta\rho^2,2\rho}} \abs{Du^\frac{m+1}{2}}^{2}\,dxdt+ \frac{1}{\rho^2}\dashint_{Q_{b\theta\rho^2,b\rho}} u^{m+1}\,dxdt
\\
&\quad\leq \frac{c({\delta,\epsilon,K})}{(b-a)^{p_o}}\bigg(\bigg(\dashint_{Q_{b\theta\rho^2,b\rho}} \abs{Du^\frac{m+1}{2}}^{2q}\,dxdt\bigg)^\frac{1}{q}
+  \rho^\frac{2}{m}\dashint_{Q_{b\theta\rho^2,b\rho}}f^\frac{m+1}{m}\,dxdt\bigg)\\
&\qquad+(\frac12+\delta)\sup_{t\in (-\theta b\rho^2,0)}\dashint_{B_{b\rho}}\frac{\abs{u-(u(t))_{B_{b\rho}}}^2}{\theta\rho^2}\,dx.
\end{align*}
Now for $\delta<\frac12$ the interpolation Lemma~6.1 of \cite{Giu03} concludes the proof.
\end{proof}

\subsection{The Non-Degenerate Regime}
In the following we assume that the opposite of $\eqref{eq:deg1}$ holds. Namely,
we consider a {cylinder} $Q_{2\theta\rho^2,2\rho}\subset E_T$, where we assume that
\begin{align}
\label{eq:ndeg} \bigg(\dashint_{Q_{2\theta\rho^2,2\rho}}\abs{u-\Mean{u}_{Q_{2\theta\rho^2,2\rho}}}^{m+1}dxdt\bigg)^\frac1{m+1}\leq \epsilon \bigg(\dashint_{Q_{2\theta\rho^2,2\rho}}u^{m+1}dxdt\bigg)^\frac1{m+1}.
\end{align}
We denote this condition the \emph{Non-Degenerate Regime}; {heuristically, it implies that $u$ is close to a solution of the linear heat equation.} 

First of all, we need the following estimate, which was originally derived in \cite{Sch13}. 
\begin{lemma}\label{lem:subcyl}
Let $N\in \setN$, {$\sigma\in(0,1)$} and consider a non-negative function $u\in L^{m+1}(Q_{s,r})$, where $Q_{s,r}=(-s,0]\times B_r$.
 If for {$\epsilon\in (0,\frac1{(\frac N{\sigma^n})^\frac{1}{m+1}+1})$}, $u$ satisfies the non-degenerate condition, i.e.
 \begin{align*}
 \bigg(\dashint_{Q_{s,r}}\abs{u- \Mean{u}_{Q_{s,r}}}^{m+1} dxdt\bigg)^\frac1{m+1}\leq \epsilon\Mean{u^{m+1}}^\frac1{m+1}_{Q_{s,r}}
 \end{align*}
 then
 \begin{align*}
\Mean{u^{m+1}}_{Q_{s,r}}^\frac{1}{m+1} 
&\leq  \frac{1}{1-((\frac N{\sigma^n})^\frac{1}{m+1}+1)\epsilon}\bigg(\dashint_{-s+\frac{k}{N}s}^{-s+\frac{k+1}{N}s}\dashint_{B_{r_1}} u^{m+1} dxdt\bigg)^\frac{1}{m+1}\\ 
&\leq \frac{(\frac N{\sigma^n})^\frac{1}{m+1}}{1-((\frac N{\sigma^n})^\frac{1}{m+1}+1)\epsilon} \Mean{u^{m+1}}_{Q_{s,r}}^\frac{1}{m+1}
 \end{align*}
for every $k\in \set{0,..,N-1}$ and $r_1\in [\sigma r,r]$.
\end{lemma}
\begin{proof}
 We use the triangular and the Jensen inequalities to find the estimate from below, namely
 \begin{align*}
   \Mean{u^{m+1}}_{Q_{s,r}}^\frac{1}{m+1}
   &\leq \bigg(\dashint_{Q_{s,r}}\abs{u- \Mean{u}_{Q_{s,r}}}^{m+1} dxdt\bigg)^\frac1{m+1}
+  \Bigabs{\Mean{u}_{Q_{s,r}}-(u)_{(-\frac{k}{N}s,-\frac{k+1}{N}s)\times B_{r_1}}}
\\
&\quad + \biggabs{\dashint_{-s+\frac{k}{N}s}^{-s+\frac{k+1}{N}s}\dashint_{B_{r_1}} u\,dxdt}\\
&\leq  \epsilon \Mean{u^{m+1}}_{Q_{s,r}}^\frac{1}{m+1}+\Bigabs{ \dashint_{-s+\frac{k}{N}s}^{-s+\frac{k+1}{N}s} \dashint_{B_{r_1}} [u-\Mean{u}_{Q_{s,r}}]\,dxdt}+ \dashint_{-s+\frac{k}{N}s}^{-s+\frac{k+1}{N}s}\dashint_{B_{r_1}} u\,dxdt
\\
&\leq \epsilon \Mean{u^{m+1}}_{Q_{s,r}}^\frac{1}{m+1}+  \dashint_{-s+\frac{k}{N}s}^{-s+\frac{k+1}{N}s} \dashint_{B_{r_1}} \abs{u-\Mean{u}_{Q_{s,r}}} dxdt + \dashint_{-s+\frac{k}{N}s}^{-s+\frac{k+1}{N}s}\dashint_{B_{r_1}} u\,dxdt
\\
&\leq \epsilon\Big(1+\Big(\frac{Nr^n}{r_1^n}\Big)^
\frac{1}{m+1}) \Mean{u^{m+1}}_{Q_{s,r}}^\frac{1}{m+1} + \dashint_{-s+\frac{k}{N}s}^{-s+\frac{k+1}{N}s}\dashint_{B_{r_1}} u\,dxdt.
\end{align*}
By absorption, we conclude the estimate from below. The estimate from above follows by enlarging the set of integration and Jensen's inequality.
\end{proof}

A classical property of diffusion equations, both linear and non-linear, is the {so-called \emph{expansion of positivity}}. The following theorem presents such a property for positive super-solutions to \eqref{IPME}-\eqref{PMD-eq: structure}. Its proof can be found in \cite[Proposition~7.1, Chapter~4]{DiBGiaVes11}.
\begin{thm}[Expansion of Positivity]
\label{thm:pos}
 Assume that $u\ge0$ is a super-solution to \eqref{IPME}-\eqref{PMD-eq: structure} in  {$(s,s+T]\times B_{2\rho}$ with $f=0$, and $T>0$ sufficiently large}. If 
 \begin{equation}\label{eq:expansion:1}
 \abs{\set{u(\cdot,s)>a}\cap B_\rho}\geq \gamma\abs{B_\rho},
 \end{equation}
for $a>0$ and $\gamma\in (0,1)$, there exist {$b>1$} and $\tilde\eta\in(0,1)$, depending only on the data and $\gamma$, such that 
\begin{equation}\label{eq:expansion:2}
u>\tilde\eta a\qquad\text{ in }\quad {(s+\frac{b}{2a^{m-1}}\rho^2,s+\frac{b}{a^{m-1}}\rho^2]\times B_{2\rho}}.
\end{equation}
\end{thm}
\noindent{The requirement on $T$ to be sufficiently large is needed to guarantee that 
\[
(s+\frac{b}{2a^{m-1}}\rho^2,s+\frac{b}{a^{m-1}}\rho^2]\times B_{2\rho}\subset (s,s+T]\times B_{2\rho}.
\]
Since $a>0$ is a priori known, and  $b>1$ depends only on the data and $\gamma$, such a condition can easily be checked.} 
In order to be able to use Theorem~\ref{thm:pos}, we need some basic consequences of~\eqref{eq:ndeg}.
\begin{lemma}
\label{lem:ndeg1}
For every $\alpha, \gamma\in(0,1)$, there exists $\epsilon\in(0,\frac12)$, such that if
\eqref{eq:ndeg} holds for this $\epsilon$, then
\begin{align}
&{\Mean{u}_{Q_{2\theta\rho^2,2\rho}}}\geq (1-\epsilon) \Mean{u^{m+1}}_{Q_{2\theta\rho^2,2\rho}}^\frac1{m+1}\label{eq:ndeg1:1}\\
&\Mean{u}_{Q_{2\theta\rho^2,2\rho}}\leq \frac{1-\epsilon}{\gamma} \dashint_{Q_{2\theta\rho^2,2\rho}}u\chi_{\{u\geq\alpha \Mean{u}_{Q_{2\theta\rho^2,2\rho}}\}}\,dxdt\label{eq:ndeg1:2}\\
&\abs{Q_{2\theta\rho^2,2\rho}\cap\set{u\geq\alpha \Mean{u}_{Q_{2\theta\rho^2,2\rho}}}}\geq \gamma^{\frac{m+1}m}\abs{Q_{2\theta\rho^2,2\rho}}\label{eq:ndeg1:3}.
\end{align}
\end{lemma}
\begin{proof}
The first statement follows directly by \eqref{eq:ndeg} and the triangular inequality. To get  \eqref{eq:ndeg1:2} we use that on the set $\set{u<\alpha \Mean{u}_{Q_{2\theta\rho^2,2\rho}}}$ we have 
 \begin{align*}
u< \Mean{u}_{Q_{2\theta\rho^2,2\rho}}\leq \frac1{1-\alpha}\abs{\Mean{u}_{Q_{2\theta\rho^2,2\rho}}-u}.
 \end{align*} 
 Combining \eqref{eq:ndeg1:1} and the previous estimate yields
\begin{align*}
 \dashint_{Q_{2\theta\rho^2,2\rho}}u\,dxdt
\leq&\dashint_{Q_{2\theta\rho^2,2\rho}}u\chi_\set{u\geq\alpha \Mean{u}_{Q_{2\theta\rho^2,2\rho}}}dxdt\\
&+\frac{1}{1-\alpha}\bigg(\dashint_{Q_{2\theta\rho^2,2\rho}}\abs{u-\Mean{u}_{Q_{2\theta\rho^2,2\rho}}}^{m+1}dxdt\bigg)^\frac{1}{m+1}\\
\leq&\dashint_{Q_{2\theta\rho^2,2\rho}}u\chi_\set{u\geq\alpha \Mean{u}_{Q_{2\theta\rho^2,2\rho}}}dxdt
+\frac{1}{1-\alpha}\frac{\epsilon}{1-\epsilon} \dashint_{Q_{2\theta\rho^2,2\rho}}u\,dxdt.
\end{align*}
If $\epsilon$ is small enough, then the last term can be absorbed and also $1-\frac{1}{1-\alpha}\frac{\epsilon}{1-\epsilon}\leq\frac{\gamma}{1-\epsilon}$, which implies  \eqref{eq:ndeg1:2}.
For the third estimate, {by \eqref{eq:ndeg1:1}, \eqref{eq:ndeg1:2}, and
 H\"older's inequality we find that} 
\begin{align*}
 \dashint_{Q_{2\theta\rho^2,2\rho}}u\,dxdt&\leq  \frac{1-\epsilon}{\gamma}\dashint_{Q_{2\theta\rho^2,2\rho}}u\chi_\set{u\geq\alpha \Mean{u}_{Q_{2\theta\rho^2,2\rho}}}dxdt
\\
&\leq \frac{1-\epsilon}{\gamma}\bigg(\frac{\abs{\set{u\geq\alpha \Mean{u}_{Q_{2\theta\rho^2,2\rho}}}}}{\abs{Q_{2\theta\rho^2,2\rho}}}\bigg)^\frac{m}{m+1}\bigg(\dashint_{Q_{2\theta\rho^2,2\rho}}u^{m+1}dxdt\bigg)^\frac{1}{m+1}\\
&\leq \frac{1}{\gamma}\bigg(\frac{\abs{\set{u\geq\alpha \Mean{u}_{Q_{2\theta\rho^2,2\rho}}}}}{\abs{Q_{2\theta\rho^2,2\rho}}}\bigg)^\frac{m}{m+1}\dashint_{Q_{2\theta\rho^2,2\rho}}u\,dxdt.
\end{align*}
This implies, that
 \begin{align*}
 \abs{Q_{2\theta\rho^2,2\rho}}\leq \frac1{\gamma^{\frac{m+1}m}} \abs{\set{u\geq\alpha \Mean{u}_{Q_{2\theta\rho^2,2\rho}}}}.
 \end{align*}
\end{proof}
\noindent Now we can apply Theorem~\ref{thm:pos} to our solution and find the following result.
\begin{proposition}
\label{pro:ndeg}
Let $u \geq 0$ be a local, weak solution to {\eqref{IPME}-\eqref{PMD-eq: structure}} for $m > 1$, and let $u$ satisfy the non-degenerate alternative in the intrinsic cylinder ${Q_{2\theta\rho^2,2\rho}}$. Then for $\alpha\in (0,\frac1{(\theta^{\frac1{m-1}}\Mean{u}_{Q_{2\theta\rho^2,2\rho}})}]$ and $\epsilon=\epsilon(\alpha)$ small enough, there exists $\eta\in (0,1)$, such that 
\begin{align*}
u(t,x)\geq \eta\Mean{u}_{Q_{2\theta\rho^2,2\rho}}\text{ for all }(t,x)\in (-\theta\rho^2,0]\times B_{2\rho}.
\end{align*}
\end{proposition}
\begin{proof}
Since $f$ is positive, $u$ is a super-solution to {\eqref{IPME}-\eqref{PMD-eq: structure}} with a vanishing right-hand side, and Theorem~\ref{thm:pos} can be applied. As before, we let $\Mean{u}_{Q_{2\theta\rho^2,2\rho}}=a$.

We choose $K_o\in[\frac1K,K]$, such that $\theta=\frac{K_o}{a^{m-1}}$. Let {$\alpha\in (0,\frac1{K_o^\frac{1}{m-1}}]$}, $\gamma\in (0,1)$ and $b>1$ accordingly chosen, by Theorem~\ref{thm:pos}. Next, we fix $\sigma\in (0,1]$, such that
\[
\frac{\sigma^2 b}{2\alpha^{m-1}K_o}=\frac14.
\]
Now, let $B_{\sigma\rho}$ be an arbitrary sub-ball of $B_\rho$. Lemma~\ref{lem:subcyl} with $N=8$ implies that for $s\in [-2\theta\rho^2,-\frac{3\theta}{4}\rho^2]$, the cylinder
$Q(s):=(-s, -s+\frac{\theta}{4}\rho^2]\times B_{\sigma\rho}$ is non-degenerate. Indeed, {it is a matter of straightforward computations to check that the analogous of \eqref{eq:ndeg} holds for $Q(s)$, with} 
$$\tilde{\epsilon}=\frac{\epsilon(\frac{8}{\sigma^n})^\frac{1}{m+1}}{1-\epsilon(1+(\frac{8}{\sigma^n})^\frac{1}{m+1})}.$$
%
\noindent The construction implies that  we can choose $\epsilon$ small enough with respect to $\sigma$, $m$, and $n$, such that $\tilde{\epsilon}$ is arbitrarily small. {In particular, once $\alpha$ and $\gamma$ have been chosen, we can select $\epsilon$, such that $\tilde\epsilon$ satisfies the assumptions of Lemma~\ref{lem:ndeg1}}. 
Therefore, by \eqref{eq:ndeg1:3}, we have that for every $s$, 
$$\abs{Q(s)\cap\set{u\geq\alpha a}}\geq \gamma^{\frac{m+1}m}\abs{Q(s)};$$ 
{in turn, this implies that there exists at least on $t\in (-s,-s+\frac{\theta}{4}\rho^2]$, such that}
\[
\abs{B_{\sigma\rho}\cap\set{u(\cdot,t)\geq{\alpha a}}}\geq \gamma^{\frac{m+1}m}\abs{B_{\sigma\rho}}.
\]
By Theorem~\ref{thm:pos} we conclude that 
 $u\geq \tilde\eta {(\alpha a)}$ {in $(-t+\frac{\theta}{4}\rho^2,-t+\frac{\theta}{2}\rho^2]\times B_{2\sigma\rho}$}. By a further application of Theorem~\ref{thm:pos}
we have
\[
u\geq \tilde\eta^2\alpha a\qquad\text{ in }\quad(-\frac{3\theta}{2}\rho^2,0]\times B_{2\sigma\rho},
\] 
as {$s$ was arbitrarily chosen in $[-2\theta\rho^2,-\frac{3\theta}{4}\rho^2]$}. The proof is concluded, since we can cover {$Q_{2\theta\rho^2,2\rho}$ with proper cylinders having cross section $B_{\sigma \rho}$ and height $\frac\theta4\rho^2$}. 
\end{proof}
\begin{remark}
As it is apparent from the proof, we have
\[
\eta=\tilde\eta^2\alpha,
\]
that is, $\eta$ depends on $\alpha$.
\end{remark}
\begin{proposition}
\label{non-deg}
Let $u \geq 0$ be a local, weak solution to {\eqref{IPME}-\eqref{PMD-eq: structure}} for $m > 1$, and let $u$ satisfy the non-degenerate condition in the intrinsic cylinder $Q_{2\theta\rho^2,2\rho}$. 
Then there exist $q\in (0,1)$ and $c>1$ depending on $\epsilon$ and the data, such that
\begin{align*}
&\dashint_{Q_{\frac{\theta}{2}\rho^2,\frac\rho2}}\abs{D u^\frac{m+1}{2}}^2\,dxdt \leq  c\bigg(\dashint_{Q_{\theta\rho^2,\rho}}\abs{D u^\frac{m+1}{2}}^{2q}dxdt\bigg)^\frac1q
+c\dashint_{Q_{\theta\rho^2,\rho}}\rho^\frac{2}{m}f\,dxdt.
\end{align*}
\end{proposition}
\begin{proof}
We estimate the right-hand side of \eqref{eq:en-est-2}. However, since we wish to use the expansion of positivity, {on the right-hand side we choose to work in the time interval $(-\theta\rho^2,0]$}. We estimate the different terms appearing in the right-hand side of \eqref{eq:en-est-2} for $\frac12\leq a'<b\leq 1$. We begin with the most difficult part. We define $a(t)=(u(t))_{B_{b\rho}}$ and $a=\Mean{u}_{\frac{b\theta}{2}\rho^2,b\rho}$. By the use of \eqref{meanchange2}, we find
 that 
\begin{align*}
&\int_{B_{b\rho}} u^{m-1}
\frac{\abs{u-(u(t))^{\eta^2}_{B_{b\rho}}}^{2}}{\rho^2}\,dx\\
&=2^{m-1}a^{m-1}\int_{B_{b\rho}}\chi_{\abs{u-a}\leq a} 
\frac{\abs{u-(u(t))^{\eta^2}_{B_{b\rho}}}^{2}}{\rho^2}\,dx +2^{m-1}\int_{B_{b\rho}}\abs{u-a}^{m-1}\frac{\abs{u-(u(t))^{\eta^2}_{B_{b\rho}}}^{2}}{\rho^2}dx
\\
&\leq (2a)^{m-1} \int_{B_{b\rho}}
\frac{\abs{u-(u(t))^{\eta^2}_{B_{b\rho}}}^{2}}{\rho^2}\,dx
+2^{m-1}\int_{B_{b\rho}}
\frac{\abs{u-(u(t))^{\eta^2}_{B_{b\rho}}}^{m+1}}{\rho^2}\,dx
\\
&\
+2^{m-1}\abs{a-(u(t))^{\eta^2}_{B_{b\rho}}}^{m-1}\int_{B_{b\rho}}\frac{\abs{u-(u(t))^{\eta^2}_{B_{b\rho}}}^{2}}{\rho^2}dx\\
&\leq (4a)^{m-1} \int_{B_{b\rho}}
\frac{\abs{u-a(t)}^{2}}{\rho^2}\,dx
+2^{m}\int_{B_{b\rho}}
\frac{\abs{u-a(t)}^{m+1}}{\rho^2}\,dx
\\
&\
+2^{m}\abs{a(t)-a}^{m-1}\int_{B_{b\rho}}\frac{\abs{u-(u(t))_{B_{b\rho}}}^{2}}{\rho^2}dx\\
&\quad+2^{m}\abs{a(t)-(u(t))^{\eta^2}_{B_{b\rho}}}^{m-1}\int_{B_{b\rho}}\frac{\abs{u-(u(t))^{\eta^2}_{B_{b\rho}}}^{2}}{\rho^2}dx\\
&
\leq (4a)^{m-1} \int_{B_{b\rho}}
\frac{\abs{u-a(t)}^{2}}{\rho^2}\,dx
+2^{m}\int_{B_{b\rho}}
\frac{\abs{u-a(t)}^{m+1}}{\rho^2}\,dx
\\
&\ 
+2^{m}\abs{a-a(t)}^{m-1}\int_{B_{b\rho}}\frac{\abs{u-a(t)}^{2}}{\rho^2}dx+c\int_{B_{b\rho}}\frac{\abs{u-a(t)}^{2}}{\rho^2}dx\int_{B_{b\rho}}{\abs{u-a(t)}^{m-1}}dx\\
&\leq (4a)^{m-1} \int_{B_{b\rho}}
\frac{\abs{u-a(t)}^{2}}{\rho^2}\,dx
+c\int_{B_{b\rho}}
\frac{\abs{u-a(t)}^{m+1}}{\rho^2}\,dx\\
&+2^{m}\abs{a-a(t)}^{m-1}\int_{B_{b\rho}}\frac{\abs{u-a(t)}^{2}}{\rho^2}dx\\
&=:(I)+(II)+(III).
\end{align*}
We integrate with respect to time, average over the domain of integration,
and estimate the three terms; by \Poincare's and Young's inequalities $(I)$ gives
\begin{align*} 
&a^{m-1}\dashint_{Q_{{b\theta}\rho^2,b\rho}}\frac{\abs{u-a(t)}^{2}}{\rho^2}dxdt\\
&\leq  a^{m-1}\sup_{t\in(-{b\theta}\rho^2,0]}\bigg(\dashint_{B_{b\rho}}\frac{\abs{u-a(t)}^{2}}{\rho^2}dx\bigg)^{1-q}\dashint_{-{b\theta}\rho^2}^0\bigg(\dashint_{B_{b\rho}}\frac{\abs{u-a(t)}^{2}}{\rho^2}dx\bigg)^{q}dt\\
&\leq \delta a^{m-1}\sup_{t\in(-{b\theta}\rho^2,0]}\dashint_{B_{b\rho}}\frac{\abs{u-a(t)}^{2}}{\rho^2}dx
+c_\delta a^{m-1}\bigg(\dashint_{Q_{{b\theta}\rho^2,b\rho}}\abs{D u}^{2q}dxdt\bigg)^\frac1q.
\end{align*}
Relying on Lemma~\ref{lem:sobpoinc}, the term coming from $(II)$ is estimated as in 
Proposition~\ref{rh:1}, and yields
\begin{align*}
&\dashint_{Q_{{b\theta}\rho^2,b\rho}}\frac{\abs{u-a(t)}^{2}}{\rho^2}dxdt
\le\delta \sup_{t\in (-{b\theta}\rho^2,0]}\dashint_{B_{b\rho}}\frac{\abs{u-(u(t))_{B_{b\rho}}}^2}{\theta\rho^2}dx\\
&\ +c(\delta,\tilde\delta)\big(\theta (u)_{Q_{\theta\rho^2,\rho}}^{m-1}\big)^{q_o}\bigg(\dashint_{Q_{{b\theta}\rho^2,b\rho}}\abs{D u^\frac{m+1}{2}}^{2\gamma}dxdt\Bigg)^\frac{1}{\gamma}.
\end{align*}
We estimate $(III)$ by using the non-degenerate condition, Lemma~\ref{lem:subcyl} with $N=1$ and $r_1=b\rho$, and H\"older's inequality.
\begin{align*}
&\dashint_{-{b\theta}\rho^2}^0\abs{a-a(t)}^{m-1}dt\dashint_{B_{b\rho}}\frac{\abs{u-a(t)}^{2}}{\rho^2}dx\\
&\quad\leq \bigg(\dashint_{-{b\theta}\rho^2}^0\abs{a-a(t)}^{m+1}dt\bigg)^\frac{m-1}{m+1}\bigg(\dashint_{-{b\theta}\rho^2}^0\bigg(\dashint_{B_{b\rho}}\frac{\abs{u-a(t)}^{2}}{\rho^2}dx\bigg)^\frac{m+1}{2}dt\bigg)^\frac{2}{m+1}\\
&\leq c\epsilon a^{m-1}\bigg(\dashint_{-{b\theta}\rho^2}^0\bigg(\dashint_{B_{b\rho}}\frac{\abs{u-a(t)}^{2}}{\rho^2}dx\bigg)^\frac{m+1}{2}dt\bigg)^\frac{2}{m+1}.
\end{align*}
We choose $q<1$, estimate by \Poincare's and Young's inequalities, use that $\frac{1}{\theta}\leq ca^{m-1}$ by Lemma~\ref{lem:subcyl}, rely on the intrinsic nature of $Q_{2\theta\rho^2,\rho}$, and take into account Proposition~\ref{pro:ndeg}, to obtain
\begin{align*}
(III)\leq &\,c\epsilon a^{m-1}\sup_{t\in(-{b\theta}\rho^2,0]}\bigg(\dashint_{B_{b\rho}}\frac{\abs{u-a(t)}^{2}}{\rho^2}dx\bigg)^\frac{m+1-2q}{m+1}\bigg(\dashint_{-{b\theta}\rho^2}^0\bigg(\dashint_{B_{b\rho}}\frac{\abs{u-a(t)}^{2}}{\rho^2}\,dx\bigg)^{q}dt\bigg)^\frac{2}{m+1}\\
\leq &\,c\epsilon a^{m-1}\sup_{t\in(-{b\theta}\rho^2,0]}\bigg(\dashint_{B_{b\rho}}\frac{\abs{u-a(t)}^{2}}{\rho^2}\,dx\bigg)^\frac{m+1-2q}{m+1}\bigg(\dashint_{-{b\theta}\rho^2}^0\dashint_{B_{b\rho}}\abs{D u}^{2q}\,dxdt \bigg)^\frac{2}{m+1}\\
\leq &\,c\epsilon (b-{a'})^2a^{m-1} \sup_{t\in(-{b\theta}\rho^2,0]}\dashint_{B_{b\rho}}\frac{\abs{u-a(t)}^{2}}{\rho^2}dx\\
&+ ca^{m-1}(b-{a'})^\frac{-(m+1-2q)}{q}\bigg(\dashint_{Q_{{b\theta}\rho^2,b\rho}}\abs{D u}^{2q}\,dxdt\bigg)^\frac1q\\
\leq &\,c\epsilon \frac{(b-{a'})^2}{\theta} \sup_{t\in(-{b\theta}\rho^2,0]}\dashint_{B_{b\rho}}\frac{\abs{u-a(t)}^{2}}{\rho^2}dx\\
&+ c(b-{a'})^\frac{-(m+1-2q)}{q}\bigg(\dashint_{Q_{{b\theta}\rho^2,b\rho}}\abs{D u^\frac{m+1}2}^{2q}\,dxdt\bigg)^\frac1q.
\end{align*}
Therefore, using \eqref{meanchange2} and the last estimates, we find that for a proper 
exponent $q_2$
\begin{align*}
&\dashint_{Q_{{b\theta}\rho^2,2\rho}}u^{m-1}\frac{\abs{u-(u(t))_{B_{b\rho}}}^{2}}{\rho^2}dxdt
\\
&\quad \leq c_\delta (b-a')^{-q_2}\bigg(\dashint_{Q_{{b\theta}\rho^2,b\rho}}\abs{D u^\frac{m+1}{2}}^{2q}\,dxdt\bigg)^\frac1q
+\frac{\delta(b-a')^2}{\theta}\sup_{t\in(-{b\theta}\rho^2,0]}\dashint_{B_{b\rho}}\frac{\abs{u-a(t)}^{2}}{\rho^2}\,dx.
\end{align*}
Proposition~\ref{pro:ndeg} implies
\begin{align*}
\int_{B_{b\rho}} \frac1{\rho^2\theta}
{\abs{u-(u(t))_{B_{b\rho}}}^{2}}\,dx
\leq c\int_{B_{b\rho}}\frac{u^{m-1}}{\rho^2}
{\abs{u-(u(t))_{B_{b\rho}}}^{2}}\,dx;
\end{align*}
therefore, we are left with the estimate of the term involving $f$. However, it can be estimated just as in Proposition~\ref{rh:1}. Hence, as in \eqref{eq:rh1} we find that
\begin{align}
\label{eq:rh2}
\begin{aligned}
&\esssup_{t\in(-a'\theta\rho^2,0]}\dashint_{\set{t}\times B_{a'\rho}}\frac{\abs{u-(u(t))_{B_{a'\rho}}}^2}{\theta\rho^2}\,dx + \dashint_{Q_{a'\theta\rho^2,a'\rho}}  u^{m-1}\abs{D u}^{2}\,dxdt \\
&\quad\leq \delta \sup_{t\in (-b\theta\rho^2,0]}\dashint_{B_{b\rho}}\frac{\abs{u-(u(t))_{B_{b\rho}}}^2}{\theta\rho^2}\,dx\\
&+c(b-a')^{-q_2-2}\big(\theta (u)_{Q_{\theta\rho^2,\rho}}^{m-1}\big)^{q_o}\bigg(\dashint_{Q_{\theta\rho^2,\rho}}\abs{D u^\frac{m+1}{2}}^{2q}\, dxdt\Bigg)^\frac{1}{q}
+c\dashint_{Q_{\theta\rho^2,\rho}}\rho^\frac{2}m f^\frac{m+1}{m}\,dxdt.
\end{aligned}
\end{align}
Since $Q_{\theta\rho^2,\rho}$ is sub-intrinsic with constant $\tilde K=2^{n+1}K$, we conclude by removing the $\delta$ term via the interpolation Lemma~6.1 of \cite{Giu03} .
\end{proof}

\section{Higher Integrability}\label{sec:int}
\subsection{Covering}
We now construct sub-intrinsic cylinders with properties convenient for our purposes.
\begin{lemma}
\label{lem:scal1}
Let $m\geq1$, $Q_{S,R}\subset E_T$, {$u\in L^{m+1}(Q_{S,R})$}, and $\hat b\in (0,2)$. For every $0<r<\rho\leq R$ there exist $s(r)$, $\theta_r$, and a cylinder $Q_{s(r),r}$ with the same center as $Q_{S,R}$, such that  the following properties hold:
\begin{enumerate}
\item\label{scal:-1}  $0\leq s(r)\leq S$ and $s(r)=\theta_r r^2$. In particular, $Q_{s(r),r}\subset E_T$.
\item\label{scal:0}  $s(r)\leq \big(\frac{r}{\rho}\big)^{\hat b}s(\rho)$, the function $s$ is continuous and strictly increasing {in} $[0,R]$. In particular, $Q_{s(r),r}\subset Q_{s(\rho),\rho}$.
 \item \label{scal:2} $\displaystyle\dashint_{Q_{s(r),r}}u^{m+1}dxdt\leq\frac{1}{\theta_r^{{\frac{m+1}{m-1}}}}$, i.e. $Q_{s(r),r}$ is sub-intrinsic.
\item \label{scal:5}  If $s(r)<\big(\frac{r}{\rho}\big)^{\hat b}s(\rho)$, then there exists $r_1\in[r,\rho)$ such that $\displaystyle\dashint_{Q_{s(r_1),r_1}}u^{m+1}dxdt=\frac{1}{\theta_{r_1}^{{\frac{m+1}{m-1}}}}$.
\item \label{scal:6} If for all $r\in (r_1,\rho)$, one has $\displaystyle\dashint_{Q_{s(r),r}}u^{m+1}dxdt<\frac{1}{\theta_r^{\frac{m+1}{m-1}}}$, then $\frac1{\theta_{r}}\leq\big(\frac{r}{\rho}\big)^{2-\hat b}\frac1{\theta_{\rho}}$ for all $r\in [r_1,\rho]$.
 \item \label{scal:3} For $\sigma\in (0,1]$, $\frac{\sigma^{2-\hat b}}{\theta_{r}}\leq\frac{1}{\theta_{\sigma r}}\leq \frac{c}{\sigma^\frac{n+2}2\theta_r}$.
\item \label{scal:4} For $\sigma\in(0,1], \abs{Q_{s(\sigma r),\sigma r}}^{-1}
\leq c\sigma^{-(n+2)(1+\frac{m-1}2)}\abs{Q_{s(r),r}}^{-1} $.
\item \label{scal:7} For $\sigma\in(0,1]$, we have $Q_{s(\sigma r),\sigma r}\subset Q_{\sigma^{\hat b}s(r),\sigma^{\frac {\hat b}2}r}$. For $a>1$ we have $Q_{as(r),ar}\subset Q_{s(\tilde{a}r),\tilde{a}r}$, for $\tilde a=\max\set{a^\frac1{\hat b},a}$.
\end{enumerate}
{The positive constant $c$ in  \eqref{scal:3} and \eqref{scal:4} depends only on the data}.
\end{lemma}
The proof of the lemma is done in \cite[Lemma~4.1]{Sch13} for the parabolic $p$-laplacian. {It can be adapted to the porous medium equation almost verbatim}, if one replaces $p$ by $m+1$ and $\abs{D u}$ by $u$.
For pure technical reasons, we will use a slightly modified construction. Indeed, we will use intrinsic cylinders of the form
\begin{align}
Q_{\theta_{r}r^2,r}(t,x):=(t-\frac{\theta_r}2 r^2,t+\frac{\theta_r}2 r^2)\times B_r(x).
\end{align}
We will also need the following result, which is stated and proved in \cite[Lemma~A.1]{Sch13}.
\begin{lemma}\label{seb-app}
Let $Q_1\subset Q$ be two cylinders and $f\in L^q(Q)$ for some $q\in[1,\infty)$. If for some $\epsilon\in(0,1)$ we have
\[
|(f)_{Q_1}|\le\epsilon [(|f|^q)_Q]^{\frac1q},
\] 
then
\[
|(f)_{Q_1}|\le\epsilon [(|f|^q)_Q]^{\frac1q}\le\frac{\epsilon}{1-\epsilon}\left(1+\left(\frac{|Q|}{|Q_1|}\right)^{\frac1q}\right)\left(\dashint|f-(f)_Q|^q\,dxdt\right)^{\frac1q}.
\]
\end{lemma}
We can finally come to the core of our argument.
\begin{lemma}
\label{lem:scal}
Let $m\geq1$, $Q_{2S,2R}\subset E_T$, and $\hat b\in (0,2)$. For every {$z\equiv(t,x)\in Q_{S,R}$} and $0<r\leq R$ there exist $s(r,z)$, $\theta_{r,z}$, and a sub-intrinsic cylinder $Q_{s(r,z),r}(z)=(t-\theta_{r,z} r^2,t+\theta_{r,z} r^2)\times B_r(x)$, such that all properties of Lemma~\ref{lem:scal1} hold.

Moreover, if we assume that
\begin{equation}\label{Eq:5:1bis}
\bigg(\dashint_{Q_{2S,2R}}u^{m+1}\,dxdt\bigg)^\frac{m-1}{m+1}\leq \frac{R^2}{S}=:\frac{C}{\theta_o},
\end{equation}
{with $C>1$}, then 
\begin{enumerate}
\item for all $z\equiv(t,x)\in Q_{S,R}$, we have 
$
\displaystyle\frac{1}{\theta_o}\leq \frac1{\theta_{R,z}}\leq C\frac{2^{\frac{2(m+1)+(m-1)n}{m-1}}}{\theta_o}.
$
\item There is a constant $c_1>1$, depending on $n$, $m$, $\hat b$ only, such that  if 
 $Q_{s(r,z),r}(z)\cap Q_{s(r,y),r}(y)\neq\emptyset$, then
$$Q_{s(r,z),r}(z)\subset Q_{s(c_1r,y),c_1r}(y)\ \ \text{ and }\ \  Q_{s(r,y),r}(y)\subset Q_{s(c_1r,z),c_1r}(z),$$ 
for each $r\leq\frac R{c_o}$, with $c_o=\max\{3^{1/{\hat b}},3\}$.
\end{enumerate}
\end{lemma}
\begin{proof} 
We directly assume \eqref{Eq:5:1bis}. For $z=(t,x)\in Q_{S,R}$, we will construct proper sub-intrinsic cylinders. We have to fix the initial cube $Q_{S(z),R}(z)$, required by Lemma~\ref{lem:scal1}. We do this, by defining 
\[
\frac{S}{2^{2(m+1)+(m-1)n}}\leq S(z)\leq S,
\] 
such that
\[
\bigg(\int_{t-S(z)}^{t+S(z)}
  \int_{B_{R}(x)}u^{m+1} \,dxdt\bigg)^{m-1}{S(z)^2}\leq R^{2(m+1)}\abs{B_{R}}^{m-1},
\]
which is possible, since
\begin{align*}
\bigg(\int_{t-S}^{t+S}
  \int_{B_{R}(x)}u^{m+1} \,dxdt\bigg)^{m-1}S^2&\leq \bigg(\int_{-2S}^{2S}
  \int_{B_{2R}}u^{m+1} \,dxdt\bigg)^{m-1}S^2\\ 
&\leq  C^{m-1}(2R)^{2(m+1)}\abs{B_{2R}}^{m-1}.
\end{align*}
Moreover, this implies that 
\[
\bigg(\dashint_{Q_{S(z),R}(z)}u^{m+1}\,dxdt\bigg)^\frac{m-1}{m+1}\leq {\frac{R^2}{2^{\frac{m-1}{m+1}} S(z)}\leq\frac{R^2}{S(z)}}=:\frac{1}{\theta_{R,z}}. 
\]
For $r\in(0,R]$ we define
\[
\tilde{s}(r,z)=\max\Bigset{s\in(0, S(z)]\,\Big|\,\Big(\int_{t-s}^{t+s}\int_{B_r(x)} u^{m+1}\,dxdt\Big)^{m-1}s^2\leq r^{2(m+1)}\abs{B_r(x)}^{m-1}},
\]
and for $\hat b\in (0,2)$
\begin{align}
\label{eq:grow}
s(r,z)=\min_{r\leq a\leq R}\Big(\frac{r}{a}\Big)^{\hat b}\tilde{s}(a,z).
\end{align}
Furthermore, we define
$\theta_{r,z}=\frac{s(r,z)}{r^2}$. By the same proof of~\cite[Lemma 4.1]{Sch13}, the functions $s(r,z)$ and $\theta_{r,z}$ satisfy all properties of Lemma~\ref{lem:scal1} stated for $s(r)$ and $\theta_r$. 

{Let us now come to (2).  We take $r\leq \frac{R}{c_o}$, where $c_o=\max\{{3^{1/{\hat b}}},3\}$. This implies that 
$$s(r,y),s(r,z)\leq \frac{S}{3}.$$ 
Without loss of generality, we may assume that $s(r,z)\geq s(r,y)$. By (8) of Lemma~\ref{lem:scal1} this implies  that
$$Q_{s(r,y),r}({y})\subset Q_{3s(r,z), 3r}(z)\subset Q_{s(c_or,z),c_or}(z).$$ 
By its definition, there is a $\rho\in[r,R]$ such that $s(r,y)=\Big(\frac{r}{\rho}\Big)^{\hat b}\tilde{s}(\rho,y)$. Moreover, since $s(r,z)\leq \Big(\frac{r}{\rho}\Big)^{\hat b} s(\rho,z)$ by (2) of Lemma~\ref{lem:scal1}, we have that $\tilde{s}(\rho,y)\leq s(\rho,z)$.
Now we let $y=(t,x)$ and $z=(t_1,x_1)$ and estimate.}

{If $\rho\in( \frac{R}{c_o} ,R{]}$, then
\begin{align*}
\Big(\frac{R}{c_o}\Big)^{2(m+1)}\abs{B_{\frac{R}{c_o}}}^{m-1}\leq \rho^{2(m+1)}\abs{B_\rho}^{m-1}&=\bigg(\int_{t-\frac{\tilde{s}(\rho,y)}2}^{t+\frac{\tilde{s}(\rho,y)}2}
  \int_{B_\rho(x)}u^{m+1} \,dxdt\bigg)^{m-1}\tilde{s}(\rho,y)^2
	\\
  &\leq \bigg(\int_{Q_{2R,2S}}u^{m+1}dxdt\bigg)^{m-1}\tilde{s}(\rho,y)^2
	\\
  &\leq \bigg(\int_{Q_{2R,2S}}u^{m+1}dxdt\bigg)^{m-1}s(\rho,z)^2
	\\
  &\leq \bigg(\int_{Q_{2R,2S}}u^{m+1}dxdt\bigg)^{m-1}S^2\\
  &\leq (2R)^{2(m+1)}\abs{B_{2R}}^{m-1},
\end{align*}
which implies that $s(\rho,z)\leq {\tilde c_1}\tilde{s}(\rho,y)$, for ${\tilde c_1}=(2c_o)^{2(m+1)+n(m-1)}$; moreover,
\[
s(r,z)\leq \Big(\frac{r}{\rho}\Big)^{\hat b}s(\rho,z)\leq {\tilde c_1}\Big(\frac{r}{\rho}\Big)^{\hat b}\tilde{s}(\rho,y)={\tilde c_1}s(r,y).
\]
On the other hand, if $\rho\in [r,\frac{R}{c_o})$, we similarly have 
\begin{align*}
\rho^{2(m+1)}\abs{B_\rho}^{m-1}&=\bigg(\int_{t-\frac{\tilde{s}(\rho,y)}2}^{t+\frac{\tilde{s}(\rho,y)}2}
  \int_{B_\rho(x)}u^{m+1} \,dxdt\bigg)^{m-1}\tilde{s}(\rho,y)^2\\
  &\leq \bigg(\int_{t_1-{\frac{3{s}(\rho,z)}2}}^{t_1+{\frac{3{s}(\rho,z)}{2}}}
  \int_{B_{3\rho}(x_1)}u^{m+1} \,dxdt\bigg)^{m-1}\tilde{s}(\rho,y)^2\\
   &\leq \bigg(\int_{t_1-\frac{{s}(3\rho,z)}{2}}^{t_1+\frac{{s}(3\rho,z)}{2}}
  \int_{B_{3\rho}(x_1)}u^{m+1} \,dxdt\bigg)^{m-1}\tilde{s}(\rho,y)^2\\
  &\leq \bigg(\int_{t_1-\frac{{s}(c_o\rho,z)}{2}}^{t_1+\frac{{s}(c_o\rho,z)}{2}}
  \int_{B_{c_o\rho}(x_1)}u^{m+1} \,dxdt\bigg)^{m-1}s(c_o\rho,z)^2\\
	&\leq (c_o\rho)^{2(m+1)}\abs{B_{c_o\rho}}^{m-1},
\end{align*}
where the last inequality follows by (8) of Lemma~\ref{lem:scal1} and the construction. 
This implies (again) that 
\[
s(r,z)\leq \Big(\frac{r}{\rho}\Big)^{\hat b}s(\rho,z)\leq {\tilde c_1}\Big(\frac{r}{\rho}\Big)^{\hat b}\tilde{s}(\rho,y)=
{\tilde c_1}s(r,y).
\]
Therefore, for $r\in (0,{\frac R{c_o}}]$, we find that 
$$Q_{s(r,z),r}(z)\subset Q_{{\tilde c_1}s(r,y), \tilde c_1r}(y) {\,\subset Q_{s(c_1r,y), c_1r}(y)},$$ 
{with $c_1={c_o}\tilde c_1$}, which finishes the proof.}
 
\end{proof}

The following version of Vitali's covering can be applied to the cylinders built in Lemma~\ref{lem:scal1}. It is inspired by \cite[Chapter~1, Lemma~1 and Lemma~2]{stein:1993}. See also \cite[Paragraph~1.5, Theorem~1]{Ev-Gar}.
\begin{lemma}\label{Lm:Vitali:1}
Let $\Omega\subset \setR^M$ and $R\in \setR$. Let there be given a two-parameter family ${\mathcal F}$ of nonempty and open sets
\[
\set{U(x,r)\,|\,x\in \Omega,\, r\in (0,R]},
\] 
which satisfy two conditions:
\begin{enumerate}
\item They are \emph{nested}, that is, 
\begin{equation}\label{Eq:nested}
\text{ for any }\ x\in\Omega,\ \text{ and }\ 0<s<r\le R,\ U(x,s)\subset U(x,r);
\end{equation}
\item There exists a constant $c_1{>1}$, such that 
\begin{equation}\label{vit-2}
U(x,r)\cap U(y,r)\neq \emptyset\ \ \Rightarrow\ \ U(x,r)\subset U(y,c_1r).
\end{equation}
\end{enumerate}
Then we can find a disjoint subfamily 
$\set{U_i}_{i\in \setN}=\set{U(x_i,r_i)}_{i\in\setN}$, such that
\[
 \bigcup\limits_{\set{x\in\Omega,\, r\in (0,\frac{R}{c_1}]}}U(x,r)\subset \bigcup_{i\in \setN} \tilde{U}_i,
 \]
 with $\tilde{U_i}=U(x_i,2c_1 r_i)$.
\end{lemma}
\begin{proof}
Without loss of generality, we may assume that $R=c_1$. Now we define the classes
\[
Q_k:=\set{U(x,r)\,|\,x\in \Omega,\, r\in [2^{-k},2^{1-k}]}\text{ for }k\in \setN.
\]
Next we take $\dot{Q}_1$ as a maximal disjoint subfamily of $Q_1$: 
by the fact that $U(x,r)$ are open and $\dot{Q}_1$ is a family of disjoint sets, we conclude that $\dot{Q}_1$ possesses at most countable many members. Indeed, let $\Omega_1:= \bigcup\limits_{\set{x\in\Omega,\, r\in (\frac12,1]}}U(x,r)$, and $\tilde\Omega_1:=\{x\in\Omega_1:\ x_i\in\setQ\ \text{ for any }\ i=1,\dots,M\}$. The set $\tilde\Omega_1$ is countable, and dense in $\Omega_1$. It is straightforward to see, that for any $x\in\tilde\Omega_1$, there exists a unique $U(y,r)\in \dot{Q}_1$ which contains $x$, and each $U(y,r)\in\dot{Q}_1$ contains at least an element of $\tilde\Omega_1$.

Next, we proceed inductively. Assuming that $\dot{Q}_1,\dot{Q}_2,\dots,\dot{Q}_{k-1}$ have already been selected, we choose $\dot{Q}_k$ to be a maximal disjoint subfamily of 
\begin{align}
\label{vit=1}
\left\{U(x,r)\in {Q}_k:\ U(x,r)\cap U(y,\rho)= \emptyset\ \text{ for all }\ U(y,\rho)\in \bigcup_{l=1}^{k-1} \dot{Q}_l\right\}.
\end{align}
Now we enumerate the members of $\bigcup_{l\in \setN} 	{\dot{Q}_k}$ and let 
$$\set{U_j}_{j\in\setN}=\set{U(x_j,r_j)}_{j\in\setN}:=\bigcup_{l\in \setN} \dot{Q}_k.$$ 
Take any $U(x,r)\in{\mathcal F}$. There exists $k\in\setN$ such that $U(x,r)\in Q_k$. By \eqref{vit=1} and the maximality of $\dot{Q}_k$, there exists $i\in \setN$, such that $r\leq 2 r_i$ and $ U(x,2r_i)\cap U(x_i,2r_i)\neq \emptyset$; we conclude by \eqref{vit-2}.
\end{proof}
From Lemma~\ref{Lm:Vitali:1}, we immediately deduce the corollary below, which might be of use also in future applications.
\begin{corollary}\label{cor:vit}
Let $\Omega\subset \setR^M$ and $R\in \setR$. Let there be given a two-parameter family ${\mathcal F}$ of nonempty and open sets
\[
\set{U(x,r)\,|\,x\in \Omega,\, r\in (0,R]},
\] 
which satisfy \eqref{Eq:nested}--\eqref{vit-2}, and the following third condition:
\begin{enumerate}
\item There exists a constant $a>1$ such that, for all $r\in (0,R]$,
\begin{equation}\label{vit-1}
0<\abs{U(x,2r)}\leq a\abs{U(x,r)}<\infty.
\end{equation}
\end{enumerate}
Then we can find a disjoint subfamily 
$\set{U_i}_{i\in \setN}=\set{U(x_i,r_i)}_{i\in\setN}$, such that
\[
 \bigcup\limits_{\set{x\in\Omega,\, r\in (0,\frac{R}{c_1}]}}U(x,r)\subset \bigcup_{i\in \setN} \tilde{U}_i,
 \]
 with $\tilde{U}_i=U(x_i,2c_1 r_i)$, $\abs{U_i}\sim \abs{\tilde{U}_i}$
 and
 \[
\abs{\Omega}\leq c\sum_i\abs{U_i},
 \]
where the constant $c>1$ depends only on $c_1$, $a$, and the dimension $M$.
\end{corollary}
In the following we will build a proper covering with respect to $u$ for the level sets of the function 
$\abs{D u^\frac{m+1}2}^2$.

We assume that we have an intrinsic cylinder
\begin{align}
\label{startcylinder}
\frac{C}{\theta_o}\geq\bigg(\dashint_{Q_{2\theta_o R^2,2R}}u^{m+1}\,dxdt\bigg)^\frac{m-1}{m+1}.
\end{align}

We start by scaling everything to the cube $Q_{2,2}$. This can be done by introducing $\gamma=\theta_o^\frac{1}{m-1}$. Then we define $\tilde{u}(y,s)=\gamma u(R y,\gamma^{m-1}R^2 s)$. For this scaled solution we find
\begin{align}
C\geq \gamma^{m+1}\dashint_{Q_{2\gamma^{m-1}R^2,2R}}u^{m+1}\,dxdt =\dashint_{Q_{2,2}}{\tilde u}^{m+1}\,dyds.
\end{align}
{Now $\tilde{u}$ is a weak solution in $Q_2$ to
\[
\tilde{u}_{s} -  \div \tilde{\A}(y,s,\tilde{u},D\tilde{u}) = \tilde{f} 
\]
with right-hand side $\tilde{f}(y,s)=\gamma^{m} f(R y,\gamma^{m-1}R^2s)$, and $\tilde{\A}$ which satisfies structure conditions analogous to \eqref{PMD-eq: structure}}.

By Lemma~\ref{lem:energy}, with the previous scaling, we find
\begin{align*}
\gamma^{m+1}R^2\dashint_{Q_{\gamma^{m-1}R^2,2R}}\abs{D u^\frac{m+1}2}^2\,dxdt&=\dashint_{Q_{1,2}}\abs{D (\tilde{u}^\frac{m+1}{2})}^2\,dyds\\
&\leq c\dashint_{Q_{2,2}}[\tilde{u}^{m+1}+\tilde{f}^\frac{m+1}{m}]\,dyds \leq C(f).
\end{align*}
We introduce the notation 
  \begin{align}
  F=\abs{D (\tilde{u}^\frac{m+1}{2})}^{2}\chi_{Q_{2,2}}.
  \end{align}
	We fix 
\begin{align}\label{b}	
\hat{b}=\frac{4}{m+1}.
\end{align}

\begin{remark}
The value of $\hat b$ is a direct consequence of the \emph{inhomogeneity} which characterizes
the energy estimates of Lemma~\ref{lem:osc-energy}. Before proceeding with the main argument, 
let us give a purely heuristic justification of \eqref{b}, using some elementary dimensional analysis. 

Denoting with the symbol $\sq{\varpi}$ the dimension of the quantity $\varpi$ (e.g. $[\tau]$ is a time, $[\rho]$ is a length), we notice that the integrands of both sides of \eqref{eq:en-est-2} have dimension
\[
\frac{\sq{u}^{2}}{[\tau]} + \frac{\sq{u}^{m+1}}{[\rho]^{2}};
\]
therefore, the homogeneity is restored once we choose a time-length $\tau$
such that
\begin{equation}\label{PMD-eq: dimension tau}
[\tau] = [u]^{1-m}[\rho]^{2}.
\end{equation}
When studying boundedness or pointwise properties of the solution, such as Harnack inequalities (see \cite{DiBGiaVes11}, and the references therein), one scales the time variable by a factor $\theta$ which is \emph{not} dimensionless: one defines
\[
\tau = \theta \rho^{2}
\]
and  relies on \emph{intrinsic cylinders} reflecting the degeneracy of the equation. The scaling parameter is
\[
\theta = \pa{\frac{u(x_{\origin}, t_{\origin})}{c}}^{1-m},
\] 
and the corresponding family of cylinders is
\[
Q_{\theta\rho^{2},\rho}(z_{\origin}) = (t_o-\theta\rho^2,t_o+\theta\rho^2)\times B_{\rho}(x_{\origin}).
\]
However, we are here interested in integrability properties of $|Du^{\frac{m+1}{2}}|$ and heuristically, the scaling parameter should now be a quantity related to the gradient $Du^{\frac{m+1}{2}}$, and not to the solution $u$.

From the relation
\[
\sq{Du^{\frac{m+1}{2}}} = \sq{u}^{\frac{m+1}{2}}[\rho]^{-1},
\]
substituting in \eqref{PMD-eq: dimension tau}, we have
\[
[\tau] = [u]^{1-m}[\rho]^{2} = \sq{u^{\frac{m+1}{2}}}^{\frac{2(1-m)}{m+1}}[\rho]^{2} = \sq{Du^{\frac{m+1}{2}}}^{\frac{2(1-m)}{m+1}}[\rho]^{2 + \frac{2(1-m)}{m+1}}.
\]
This relation shows that, when dealing with the higher integrability of the gradient, cumbersome how it may look, the time scaling one ends up working with
is
\[
\tau \approx \theta \rho^{\frac{4}{m+1}} \qquad \text{with} \qquad \sq{\theta} = \sq{Du^{\frac{m+1}{2}}}^{\frac{2(1-m)}{m+1}},
\]
and the exponent of $\rho$ is precisely the value of $\hat b$ given in \eqref{b}.

\end{remark}
We will apply Lemma~\ref{lem:scal} with respect to $\tilde{u}$ and this choice of $\hat{b}$ on the sub-intrinsic initial cylinder $Q_{2,2}$. To simplify the notation, for $y\in Q_{1,1}$ and $r\in (0,1)$ 
we denote the sub-intrinsic cube $\displaystyle Q_{s(r,y),r}(y)$ defined in Lemma~\ref{lem:scal} with $\displaystyle Q(r,y)$.  
\begin{lemma}
\label{lem:bigcubes}
Fix {$c_2\in [1,\infty)$}.
For $\frac12\leq a<b\leq 1$, there exists a parameter $\lambda$, defined by
\[
\lambda_{a,b}=\frac{C(f)}{\abs{b-a}^{\tau}},
\]
with $\tau$ depending only on the data, such that, if for $z\in Q_{a,a}$ and the corresponding cube $Q(\rho,z)$, the intersection $Q(c_2\rho,z)\cap Q_{b,b}^c$ is not empty, then
\[
\dashint_{Q(\rho,z)}F\,dxdt\leq \lambda_{a,b}.
\]
Here $\displaystyle C(f)=c\dashint_{Q_{2,2}}[\tilde{f}^\frac{m+1}{m}+\tilde{u}^{m+1}]\,dxdt$, with $c>0$ depending only on the data.
\end{lemma} 
\begin{proof}
If $Q(c_2\rho,z)\cap Q_{b,b}^c\neq \emptyset$, then either
$c_2\rho>(b-a)$, or $s(c_2\rho,z)>(b-a)$. In the latter case, by \eqref{scal:0} of Lemma~\ref{lem:scal1}, we have that $(b-a)<c\rho^{\hat b}$. This implies that for $\tilde{b}=\min\set{1,\hat b}$ in any case $c\rho^{\tilde{b}}>(b-a)$. Therefore, \eqref{scal:3} of Lemma~\ref{lem:scal1} implies
 \begin{align}
\label{eq:subintr}
\dashint_{Q(\rho,z)} F\,dxdt\leq\frac{C(f)}{\theta_{\rho,z}\abs{b-a}^{(n+2)/{\tilde b}}}\leq \frac{C(f)}{\theta_{1,z}\abs{b-a}^{({n+2}+\frac{n+2}{2})/{\tilde b}}}\leq \frac{C(f)}{\abs{b-a}^{\tau}},
 \end{align}
where $\theta_{1,z}$ is defined via Lemma~\ref{lem:scal}. This defines $\tau$. 
\end{proof}
Now, we introduce the related intrinsic maximal function
\[
\Mi(g)(z)=\sup_{Q(r,y)\ni z,\ y\in Q_{1,1}}\dashint_{Q(r,y)}\abs{g}\,dxdt,
\]
and we define 
\begin{align}
M^*(g)(t,x):=\sup_{t\in I\subset (-2,2),\ x\in B\subset B_2}\dashint_I\dashint_B\abs{g}\,dxdt.
\end{align}
Notice that we have 
\begin{align}
\abs{g}\leq\Mi(g)\leq M^*(g)\qquad \text{ a.e.};
\end{align}
therefore, $\Mi$ is continuous from $L^q\to L^q$, whenever $g\in L^q$. 

We define the level sets of $F$ by 
\begin{align}
O_\lambda :=\set{\Mi(F)>\lambda}.
\end{align} 
	
The next proposition is the core of the proof of the higher integrability. It constructs a covering, which allows to exploit the reverse H\"older estimates of the previous section in a suitable way. It is a covering of Calderon-Zygmund type for $F$, build using cylinders scaled with respect to $\tilde u$. 

The fact that the scaling is done with respect to $\tilde u$, and not with respect to the function whose level sets are covered, i.e. $F$, makes things quite delicate. In this context, this seems the right way to proceed, instead of relying on the by-now standard approach of parabolic intrinsic Calderon-Zygmund covering, originally introduced by Kinnunen \& Lewis~\cite{Kinnunen:2000}.   
	
	\begin{proposition}
	\label{lem:czcover}
	Let  
	$\frac12\leq a<b\leq 1$ and $\lambda_{a,b}$ the quantity defined in Lemma~\ref{lem:bigcubes} for a proper choice of the parameter $c_2$. Let $q\in (0,1)$ be the exponent defined by the reverse H\"older estimates of Propositions~\ref{deg} and \ref{non-deg}.
	
For every $\lambda>\lambda_{a,b}$, and every $z\in O_\lambda\cap Q_{a,a}$, there exist 
$Q_z\subset Q_z^*\subset Q_z^{**}\subset Q_{b,b}$, which satisfy the following properties:
\begin{enumerate}
\item $\abs{Q_z},\abs{Q_z^*},\abs{Q_z^{**}}$ are of comparable size.
\item For any $y\in Q_z^*$ 
\[
\lambda\leq {c}\dashint_{Q_{z}}F\,dxdt\leq c\bigg(\dashint_{Q_{z}^*}F^q\,dxdt \bigg)^\frac1q + cM^*(\tilde{f}^\frac{m+1}{m})(y).
\]
\item $\displaystyle\dashint_{Q_z^{**}} F\,dxdt\leq 2 \lambda$.
%
%
\end{enumerate}
Moreover, the set $O_\lambda\cap Q_{a,a}$ can be covered by a family of cylinders 
$Q_i^{**}:=Q_{z_i}^{**}\subset Q_{b,b}$, such that the cylinders of the family $Q_i^*$ are pairwise disjoint.
\end{proposition}
\begin{proof}
Let $c_2\in[1,\infty)$ the parameter introduced in Lemma~\ref{lem:bigcubes}: choose $c_2=4c_1$, where $c_1$ is the constant determined in Lemma~\ref{lem:scal}, and let $\lambda_{a,b}$ be the corresponding quantity defined in Lemma~\ref{lem:bigcubes}. We fix $\lambda>\lambda_{a,b}$, and for $z\in O_\lambda\cap Q_{a,a}$ we choose $y_z$ and $r_z$, such that $Q(r_z,y_z)\ni z$ and
\begin{align}\label{eq:lambda}
\lambda < \dashint_{Q(r_z,y_z)}F\,dxdt\ \ \text{ and }\ \ \dashint_{Q(\rho,\xi)}F\,dxdt\leq 2\lambda
\end{align}
for all $Q{(\rho,\xi)}\supset Q(r_z,y_z)$; such a cylinder certainly exists by the very definition of the set $O_\lambda$. 
In the following notation, for $c\in \setR^+$, we let $\tilde{c}=\max\set{c^{1/\hat b},c}$.
According to the following table we will carefully choose cylinders, for $z\in O_\lambda\cap Q_{a,a}$.
\vskip.2truecm
\begin{easylist}
\ListProperties(Hide=100, Hang=true, Progressive=3ex, Style*=$\bullet$ ,
Style2*=$\circ$ ,Style3*=\tiny$\blacksquare$ )
\item {\bf Case 1}: There exists $\rho\in [4r_z,8r_z]$ such that  $Q(\rho,y_z)$ is intrinsic and the degenerate alternative holds in $Q(\rho,y_z)$. Then we let\\
\begin{itemize}
\item $Q_z=Q(\rho, y_z),\qquad Q_z^{*}=2Q(\rho,y_z),\qquad Q_z^{**}={Q(4c_1 \rho, y_z)}$\\
\end{itemize}
\item {\bf Case 2}:  There exists $\rho\in [4r_z,8r_z]$ such that  $Q(\rho,y_z)$ is intrinsic and the non-degenerate alternative holds in $Q(\rho,y_z)$. Then we let\\
\begin{itemize}
\item $Q_z=\frac14Q(\rho, y_z),\qquad Q_z^{*}=\frac12Q(\rho,y_z),\qquad Q_z^{**}=Q(2c_1\rho, y_z)$\\
\end{itemize}
\item {\bf Case 3}: There exists no $\rho\in [4r_z,8r_z]$ such that  $Q(\rho,y_z)$ is intrinsic. Then we let\\
\begin{itemize}
\item $Q_z=Q(4r_z, y_z), \qquad Q_z^{*}=2Q(4r_z,y_z),\qquad Q_z^{**}={Q(16c_1 r_z, y_z)}$\\
\end{itemize}
\end{easylist}
The choice of $c_2$  implies that $Q^{**}_z\subset Q_{b,b}$ in all the above cases, as otherwise there is a contradiction to $\lambda>\lambda_{a,b}$ by \eqref{eq:lambda}. 

On the one hand, \eqref{scal:7} of Lemma~\ref{lem:scal1} implies that $Q_z^{**}\supset Q_z\supset Q(r_z,y_z)$, and therefore we conclude that
\[
\dashint_{Q_z^{**}} F\,dxdt\leq 2 \lambda.
\]
On the other hand, since \eqref{scal:4} of Lemma~\ref{lem:scal1} implies that
$\abs{Q_z}\approx \abs{Q_z^{**}}\approx \abs{Q(r_z,y_z)}$, by \eqref{eq:lambda} we find that
\begin{align}
\label{lambda:low}
\lambda \leq c \dashint_{Q_z} F\,dxdt.
\end{align}
The proof of the reverse H\"older inequality has to be split in several sub-cases.
\vskip.2truecm
\noindent{\bf Case 1.}\\
In this case $Q_z$ is intrinsic and the degenerate alternative holds. In order to apply Proposition~\ref{deg}, we have to check that $2Q_z$ is sub-intrinsic. To prove this, firstly observe  $2Q_z\subset Q(\tilde{2}\rho,y_z)$, which is sub-intrinsic by construction. Secondly, since $2Q$ and $Q(\tilde{2}\rho,y_z)$ have comparable measure by \eqref{scal:4} of Lemma~\ref{lem:scal1}, we got that $2Q_z$ is sub-intrinsic. Consequently, \eqref{lambda:low} and Proposition~\ref{deg} imply
\[
\lambda \leq c\dashint_{Q_z}  F\,dxdt\leq  c\bigg(\dashint_{Q_z^*}F^q\,dxdt\bigg)^\frac1q+ cM^*(\tilde{f}^\frac{m+1}{m})(y),
\]
 for any $y\in Q_z^*$.
\vskip.2truecm
\noindent{\bf Case 2.}\\ 
 In this case $4Q_z$ is intrinsic and the non-degenerate alternative holds. Proposition~\ref{non-deg} and \eqref{lambda:low} directly imply
\[
\lambda \leq c\dashint_{Q_z}  F\,dxdt\leq  c\bigg(\dashint_{Q_z^*}F^q\,dxdt\bigg)^\frac1q+ M^*(\tilde{f}^\frac{m+1}{m})(y),
\]
 for any $y\in Q_z^*$.
\vskip.2truecm
\noindent{\bf Case 3.}\\
This is the most delicate part.
We begin by applying Proposition~\ref{rh:1} on the cylinder $Q_z$, with $\delta=\tilde\delta\in(0,1)$ to be chosen later. Together with \eqref{lambda:low}, we find
\[
\lambda \leq c\dashint_{Q_z}  F\,dxdt\leq  c\bigg(\dashint_{Q_z^*}F^q\,dxdt\bigg)^\frac1q+ M^*(\tilde{f}^\frac{m+1}{m})(y)+\frac{2\delta}{(4r_z)^2}\Big(\Mean{u^{m+1}}_{Q_z^*}+\frac1{2\theta_{4r_z,y_z}^\frac{m+1}{m-1}}\Big).
\]
 By \eqref{scal:2}, \eqref{scal:3}, \eqref{scal:7} of Lemma~\ref{lem:scal1} we find that 
\[
\Mean{\tilde u^{m+1}}_{Q_z^*}\leq c \Mean{\tilde u^{m+1}}_{Q(\tilde{2}\cdot4r_z,y_z)}\leq \frac{c}{\theta_{\tilde{2}\cdot 4r_z,y_z}^\frac{m+1}{m-1}}\leq \frac{c}{\theta_{4r_z,y_z}^\frac{m+1}{m-1}}.
\]
 This implies that
\begin{align}\label{subintr:1}
\lambda \leq c\dashint_{Q_z}  F\,dxdt\leq  c\bigg(\dashint_{Q_z^*}F^q\,dxdt\bigg)^\frac1q+ M^*(\tilde{f}^\frac{m+1}{m})(y)+\frac{c\delta}{(4r_z)^2\theta_{4r_z,y_z}^\frac{m+1}{m-1}}.
\end{align}
 In the following we will show that
\begin{align}
\label{subintr:3}
\frac{1}{(4r_z)^2\theta_{4r_z,y_z}^\frac{m+1}{m-1}}\leq c\lambda+ M^*(\tilde{f}^\frac{m+1}{m})(y).
\end{align}
Once this is proven, then by absorption the result follows from \eqref{subintr:1}.

We begin by defining $\rho_z:=\inf\set{r\in [ 4r_z,1] : Q(r,y_z)}$ is intrinsic. Since we are in Case 3, we find that $\rho_z\in (8r_z,1]$; \eqref{scal:6} of Lemma~\ref{lem:scal1}, the choice of $\hat{b}$ in \eqref{b}, and the previous fact together imply that
\begin{align}
\label{subintr:2}
\frac{1}{(4r_z)^2\theta_{4r_z,y_z}^\frac{m+1}{m-1}}\leq \frac{1}{\rho_z^2\theta_{\rho_z,y_z}^\frac{m+1}{m-1}}.
\end{align}
In the simple case $\rho_z\in (\frac{1}{\tilde 2},1]$, we further estimate by \eqref{scal:3} of 
Lemma~\ref{lem:scal1}
\[
\frac{1}{(4r_z)^2\theta_{4r_z,y_z}^\frac{m+1}{m-1}}\leq c\leq \lambda,
\]
by the choice of $\lambda_{a,b}$.

In the difficult case $\rho_z\in (8r_z,\frac{1}{\tilde 2}]$, since $Q(r,y_z)$ is intrinsic, by \eqref{subintr:2} we find  that
\[
\frac{c}{(4r_z)^2\theta_{4r_z,y_z}^\frac{m+1}{m-1}}\leq c\dashint_{Q(\rho_z,y_z)}\frac{\tilde{u}^{m+1}}{\rho_z^2}\,dxdt.
\]
In order to apply Proposition~\ref{deg} on the cylinder $Q(\rho_z,y_z)$, which is intrinsic, we have to show that $2Q(\rho_z,y_z)$ is sub-intrinsic and that in $Q(\rho_z,y_z)$ the degenerate alternative holds. 
We first show that in $Q(\rho_z,y_z)$ the degenerate condition is satisfied: this is due to the fact, that by the choice of $\rho$ we find that $Q(r,y_z)$ is strictly sub-intrinsic for all $r\in [\frac12{\rho_z},\rho_z)$. Now 
\eqref{scal:3}, \eqref{scal:6}, \eqref{scal:5} of Lemma~\ref{lem:scal1} imply
\[
\Mean{\tilde u^{m+1}}^\frac{m-1}{m+1}_{Q(\frac12{\rho_z},y_z)}\leq \frac{1}{\theta_{\frac12{\rho_z},y_z}}\leq \Big(\frac12\Big)^{2-\hat{b}}\frac{1}{\theta_{\rho_z,y_z}}\leq \Big(\frac12\Big)^{2-\hat{b}}\Mean{\tilde u^{m+1}}^\frac{m-1}{m+1}_{Q({\rho_z},y_z)}.
\]
Lemma~\ref{seb-app} implies that in $Q({\rho_z},y_z)$ the degenerate alternative holds. 

That $2Q(\rho_z,y_z)$ is sub-intrinsic follows again by the fact that $2Q(\rho_z,y_z)$ has size comparable to $Q(\tilde 2\rho_z,y_z)$, which is sub-intrinsic and a superset of $2Q(\rho_z,y_z)$ by \eqref{scal:4}, \eqref{scal:2}, \eqref{scal:7} of Lemma~\ref{lem:scal1}.

Finally Proposition~\ref{deg}, Jensen's inequality and \eqref{eq:lambda} imply that
\begin{align*}
\frac{c}{(4r_z)^2\theta_{4r_z,y_z}^\frac{m+1}{m-1}}&\leq c\dashint_{Q(\rho_z,y_z)}\frac{\tilde{u}^{m+1}}{\rho_z^2}\,dxdt\leq c\dashint_{Q(\rho_z,y_z)} F\,dxdt +c M^*(\tilde{f}^\frac{m+1}{m})(y)\\
&\leq c\lambda+c M^*(\tilde{f}^\frac{m+1}{m})(y).
\end{align*}
This concludes the construction of $Q_z,Q_z^*$ and $Q_z^{**}$.

The covering is gained by applying Corollary~\ref{cor:vit}. Indeed, since
\[
O_\lambda\cap Q_{a,a}\subset \bigcup_{z\in 	O_\lambda\cap Q_{a,a}} Q^*_z,
\]
due to Lemma~\ref{lem:scal} and (\ref{scal:4}) of Lemma~\ref{lem:scal1}, Corollary~\ref{cor:vit} implies that we can find an at most  countable sub-family of disjoint cylinders $Q^*_i=Q^*_{z_i}$, such that 
\[
O_\lambda\cap Q_{a,a}\subset \bigcup_{i} Q^{**}_i,
\]
which concludes the proof.
\end{proof}

We can finally conclude and prove the higher integrability result.

\begin{thm}[Intrinsic]
\label{thm-intr}
Let $u\ge0$ be a local, weak solution to {\eqref{IPME}-\eqref{PMD-eq: structure}} in the space-time cylinder $E_T$ for $m>1$. There exist an exponent $p>1$ and a constant $c$ that depend only on the data, such that for any sub-intrinsic parabolic cylinder 
\[
\frac{C}{\theta_o}\geq\bigg(\dashint_{Q_{2{{\theta_o}} R^2,2R}} u^{m+1}\,dxdt\bigg)^\frac{m-1}{m+1},
\]
we have
\begin{align*}
\bigg(\dashint_{Q_{\frac12\theta_o R^2,\frac12R}}\abs{D u^\frac{m+1}{2}}^{2p}dxdt\bigg)^\frac{m-1}{p(m+1)}\leq c\bigg(\dashint_{Q_{2\theta_o R^2,2R}}\frac{f^\frac{p(m+1)}{m}}{R^{2p}}\,dxdt\bigg)^\frac{m-1}{p(m+1)}+\frac{c'}{R^\frac{2(m-1)}{m+1}\theta_o},
\end{align*}
where $c'$ additionally depends on $C$.
\end{thm}
\begin{proof}
We define the so-called \emph{bad set}
\begin{align*}
U_\lambda :=O_\lambda\cap \set{M^*(\tilde f^\frac{m+1}{m})\chi_{Q_{2,2}}\leq \tilde\epsilon\lambda},
\end{align*}
for some $\tilde\epsilon$, which will be chosen later.
We proceed by providing a re-distributional estimate. We take $\frac12\leq a<b\leq 1$ and the corresponding covering constructed in Lemma~\ref{lem:czcover}.
We start by
  \begin{align*}
  \abs{Q_i}=\abs{Q_i\cap U_\lambda}+\abs{Q_i\cap(U_\lambda)^c}.
  \end{align*}

Let us first consider the case 
\[
\frac{\abs{Q_i\cap U_\lambda}}{\abs{Q_i}}\geq \frac12.
\] 
This implies that there exists $y\in  Q_i$, such that $M^*(\tilde{f}^\frac{m+1}{m})(y)\leq \tilde\epsilon \lambda$.
	We can apply the reverse 
  H\"older estimate of Lemma~\ref{lem:czcover}, and obtain for some $\gamma\in(0,1)$ that
  \begin{align*}
  \lambda^q\leq c\left(\dashint_{Q_i} F\,dxdt\right)^q
  \leq \frac{c}{\abs{Q_i}}\int_{Q_i^{*}}F^q\chi_{\set{F>\gamma\lambda}}\,dxdt+c(\gamma\lambda)^q+c\tilde\epsilon\lambda^q.
  \end{align*}
We now choose $\gamma$, and $\tilde\epsilon$ conveniently small, such that $c(\gamma\lambda)^q+c\tilde\epsilon\lambda^q=\frac12\lambda^q$ and find
 \begin{align*}
\lambda\abs{Q_i}\leq c\lambda^{1-q}\int_{Q_i^{*}}F^q\chi_{\set{F>\gamma\lambda}}\,dxdt.
  \end{align*}
On the other hand,
\[
\abs{Q_i}\leq 2\abs{Q_i\cap (U_\lambda)^c}\quad\Rightarrow\quad\lambda\abs{Q_i}\leq 2\lambda\abs{Q_i\cap (U_\lambda)^c}.
\] 
Therefore, in any case,
\[
\lambda\abs{Q_i}\leq c\lambda^{1-q}\int_{Q_i^{*}}F^q\chi_{\set{F>\gamma\lambda}}\,dxdt+2\lambda\abs{Q_i\cap (U_\lambda)^c}.
\]
We proceed by using the last estimates as well as the fact that $(Q_i^{**})_i$ covers the set $O_\lambda\cap Q_{a,a}$.
  \begin{align*}
  \int_{Q_{a,a}\cap O_\lambda} F\,dxdt&\leq \sum_i\int_{Q_i^{**}}F\,dxdt\leq 2\lambda\sum_i\abs{Q_i^{**}}\leq c\lambda\sum_i\abs{Q_i}\\
  &\leq c\lambda^{1-q}\sum_i\int_{Q_i^*}F^q\chi_{\set{F>\gamma\lambda}}\,dxdt+2\lambda\abs{Q_i\cap(U_\lambda)^c}\\
  &\leq c\lambda^{1-q}\int_{Q_{b,b}}F^q\chi_{\set{F>\gamma\lambda}}\,dxdt+2\lambda\abs{Q_{b,b}\cap \set{M^*(\tilde{f}^\frac{m+1}{m}\chi_{Q_{2,2}})>\tilde\epsilon\lambda}}.
  \end{align*}
  We pick $\alpha\in (0,1)$ to be chosen later, $k\in\setN$, multiply the above estimate by $\lambda^{-\alpha}$, and integrate from $\lambda_{a,b}$ to $k$ with respect to $\lambda$. This implies
\begin{align*}
 (I)&:= \int_{\lambda_{a,b}}^k\lambda^{-\alpha}\int_{Q_{a,a}\cap O_\lambda} F\,dxdt\,d\lambda
  \leq c\int_{\lambda_{a,b}}^k\lambda^{1-q-\alpha}\int_{Q_{b,b}}F^q\chi_{\set{F>\gamma\lambda}}\,dxdt\,d\lambda\\
  &\quad+2\int_{\lambda_{a,b}}^k\lambda^{1-\alpha}\abs{Q_{b,b}\cap\set{M^*(\tilde{f}^\frac{m+1}{m})>\tilde\epsilon\lambda}}d\lambda
 =:(II)+(III).
  \end{align*}
We estimate from below
  \begin{align*}
(I)&\geq \int_0^k\lambda^{-\alpha}\int_{Q_{a,a}} F\chi_{\set{F>\lambda}}\,dxdt\,d\lambda-\frac{\lambda_{a,b}^{2-\alpha}}{1-\alpha}\\
&=\int_{Q_{a,a}}F\int_0^{\min\set{F(x),k}}\lambda^{-\alpha}d\lambda\,dxdt-\frac{\lambda_{a,b}^{2-\alpha}}{1-\alpha}\\
&=\frac{1}{1-\alpha}\int_{Q_{a,a}} F\min\set{F,k}^{1-\alpha}\,dxdt -\frac{\lambda_{a,b}^{2-\alpha}}{1-\alpha}.
  \end{align*}
The bound from above is analogous
  \begin{align*}
  (II)&\leq c\int_{\gamma{\lambda_{a,b}}}^{\gamma k} \lambda^{1-q-\alpha}\int_{Q_{b,b}}F^q\chi_{\set{F>\lambda}}\,
  dxdt\,d\lambda\\
  &\leq \frac{c}{2-q-\alpha}\int_{Q_{b,b}} F^q\min\set{F,k}^{2-q-\alpha}\,dxdt.
  \end{align*}
Finally, $(III)$ is estimated by the continuity of the maximal function and the classical integral representation via level sets. We calculate and estimate
  \begin{align*}
  (III)&= \int_{\lambda_{a,b}}^k\lambda^{1-\alpha}\abs{Q_{b,b}\cap\set{M^*(\tilde{f}^\frac{m+1}{m})>\tilde\epsilon\lambda}}d\lambda
	\leq c\int_{Q_{b,b}} \abs{M^*(\tilde{f}^\frac{m+1}{m}\chi_{Q_{2,2}})}^{2-\alpha}\,dxdt\\
	&\leq c\int_{Q_{2,2}} \tilde{f}^\frac{(m+1)(2-\alpha)}{m}\,dxdt.
  \end{align*}
  All together, using the definition of $\lambda_{a,b}$ from \eqref{eq:subintr}, we find that
  \begin{align*}
  \frac{1}{1-\alpha}\int_{Q_{a,a}} F\min\set{F,k}^{1-\alpha}\,dxdt& \leq \frac{c}{2-q-\alpha}\int_{Q_{b,b}} F^q\min\set{F,k}^{2-q-\alpha}\,dxdt \\
  &\quad+ \frac{c\abs{b-a}^{-\tau(2-\alpha)}}{1-\alpha}
   + c\int_{Q_{2,2}} \tilde{f}^\frac{(m+1)(2-\alpha)}{m}\,dxdt
  \end{align*}
  Now, we fix $\alpha\in (0,1)$ in such a way, that
  \[
  \frac{c(1-\alpha)}{2-q-\alpha}\leq \frac12.
  \]
  This implies
    \begin{align*}
 \int_{Q_{a,a}} F\min\set{F,k}^{1-\alpha}\,dxdt& \leq \frac{1}{2}\int_{Q_{b,b}} F\min\set{F,k}^{1-\alpha}\,dxdt \\
  &\quad+ c\abs{b-a}^{-\tau(2-\alpha)}
   + c\int_{Q_{2,2}} \tilde{f}^\frac{(m+1)(2-\alpha)}{m}\,dxdt.
  \end{align*}
  Finally, the interpolation Lemma~6.1 of \cite{Giu03} implies that for every $k\in \setN$
  \begin{align*}
 \int_{Q_{\frac12,\frac12}} F\min\set{F,k}^{1-\alpha}\,dxdt& \leq c
   + c\int_{Q_{2,2}} \tilde{f}^\frac{(m+1)(2-\alpha)}{m}\,dxdt.
  \end{align*}
Letting $k\to\infty$ for $p=2-\alpha>1$ yields that
	\begin{align*}
\dashint_{Q_{\frac12,\frac12}}\abs{D \tilde{u}^\frac{m+1}{2}}^{2p}dxdt\leq  c
   + c\int_{Q_{2,2}} \tilde{f}^\frac{(m+1)(2-\alpha)}{m}\,dxdt.
\end{align*}
This implies the desired result by scaling back to $u$.
 \end{proof}
	\begin{thm}[parabolic]
\label{thm-para}
Let $u\ge0$ be a local, weak solution to {\eqref{IPME}-\eqref{PMD-eq: structure}} in the space-time cylinder $E_T$ for $m>1$. There exist an exponent $p>1$ and a constant $c$, depending only on the data, such that for any parabolic cylinder $Q_{R^2,R}\subset E_T$ with
\[
{\bigg(\dashint_{Q_{R^2,R}} u^{m+1}\,dxdt\bigg)^{m-1}=K,}
\]
we have
\begin{align}\label{eq:par-final}
\bigg(\dashint_{\frac12Q_{R^2,R}}\abs{D u^\frac{m+1}{2}}^{2p}\,dxdt\bigg)^\frac{m-1}{p(m+1)}\leq c\sqrt{K}\bigg(\dashint_{Q_{R^2,R}}\frac{f^\frac{p(m+1)}{m}}{R^{2p}}\,dxdt\bigg)^\frac{m-1}{p(m+1)}+cK^\frac{3}{2}+c.
\end{align}
\end{thm}
\begin{proof}
Estimate \eqref{eq:par-final} is proved by covering $Q_{R^2,R}$ with proper sub-intrinsic cylinders. As {\eqref{IPME}-\eqref{PMD-eq: structure}} is essentially invariant under the classical parabolic scaling, without loss of generality, we may assume $R=1$ .

We define $K$ as the number for which
\[
\bigg(\int_{Q_{1,1}} u^{m+1}\, dxdt\bigg)^{m-1}=K.
\]
Now let $r\in (0,1]$ and $s\in (0,1]$;  any $Q_{s,r}\subset Q_{1,1}$ satisfies
\[
\bigg(\int_{Q_{s,r}} u^{m+1} \,dxdt\bigg)^{m-1}\leq K.
\]
If $\frac{r^{2(m+1)}\abs{B_{r}}^{m-1}}{s^2}\geq K$, then the cylinder is sub-intrinsic, since
\[
\bigg(\int_{Q_{s,r}}u^{m+1} \,dxdt\bigg)^{m-1}\leq \frac{r^{2(m+1)}\abs{B_{r}}^{m-1}}{s^2}.
\]  
If $K\leq 1$, we can pick $s,r=1$ and the result follows by Theorem~\ref{thm-intr}. If $K>1$, we choose $r=1$ and $s=\frac{1}{\sqrt{K}}$. We can then cover $Q_{1,1}$ by  $N$ sub-cylinders of the above type, where 
$$\left\lfloor\sqrt{K}\right\rfloor\leq  N \leq \left\lceil\sqrt{K}\right\rceil,$$
and $\lfloor\cdot\rfloor$, $\lceil\cdot\rceil$ are the floor and ceiling functions, respectively.
This concludes the proof by Theorem~\ref{thm-intr}.
\end{proof}

\providecommand{\bysame}{\leavevmode\hbox to3em{\hrulefill}\thinspace}
\providecommand{\MR}{\relax\ifhmode\unskip\space\fi MR }
\providecommand{\MRhref}[2]{%
  \href{http://www.ams.org/mathscinet-getitem?mr=#1}{#2}
}
\providecommand{\href}[2]{#2}

\end{document}